\documentclass[a4,11pt,fleqn]{article}
\usepackage{pstricks}
\usepackage{graphics}
\usepackage{enumitem}
\usepackage{amsmath}
\usepackage{amssymb}
\usepackage{pifont}
\usepackage{theorem} 
\usepackage{euscript}
\usepackage{exscale,relsize}
\usepackage{epic,eepic}
\usepackage{multicol}
\usepackage{makeidx}
\usepackage{srcltx}         
\usepackage{xcolor}
\usepackage{url}
\usepackage{charter}
\usepackage{mathrsfs} 
\newcommand{\email}[1]{\href{mailto:#1}{\nolinkurl{#1}}}
\usepackage[normalem]{ulem}     
\newlength{\mySubFigSize}
\newtheorem{theorem}{Theorem}[section]
\newtheorem{lemma}[theorem]{Lemma}

\newtheorem{proposition}[theorem]{Proposition}
\newtheorem{assumption}[theorem]{Assumption}
\theoremstyle{plain}{\theorembodyfont{\rmfamily}%
}
\theoremstyle{plain}{\theorembodyfont{\rmfamily}%
}
\theoremstyle{plain}{\theorembodyfont{\rmfamily}%
\newtheorem{remark}[theorem]{Remark}}
\theoremstyle{plain}{\theorembodyfont{\rmfamily}%
\newtheorem{Algorithm}[theorem]{Algorithm}}
\theoremstyle{plain}{\theorembodyfont{\rmfamily}%
}
\theoremstyle{plain}{\theorembodyfont{\rmfamily}%
}
\theoremstyle{plain}{\theorembodyfont{\rmfamily}%
\newtheorem{definition}[theorem]{Definition}}
\theoremstyle{plain}{\theorembodyfont{\rmfamily}%
}

\numberwithin{equation}{section}
\definecolor{labelkey}{rgb}{0,0.08,0.45}
\definecolor{refkey}{rgb}{0,0.6,0.0}
\definecolor{Brown}{rgb}{0.45,0.0,0.05}
\definecolor{dgreen}{rgb}{0.00,0.49,0.00}
\definecolor{dblue}{rgb}{0,0.08,0.75}
\RequirePackage[dvips,colorlinks,hyperindex]{hyperref}
\hypersetup{linktocpage=true,citecolor=dblue,linkcolor=dgreen}
\usepackage{pstricks-add}
\usepackage{pst-grad}
\usepackage{pst-plot} 
\usepackage{pst-eucl} 
\numberwithin{equation}{section}

\DeclareMathOperator{\Argmin}{\mathrm{Argmin}}
\DeclareMathOperator{\argmin}{\mathrm{argmin}}

\renewcommand{\leq}{\ensuremath{\leqslant}}
\renewcommand{\geq}{\ensuremath{\geqslant}}

\newcommand{\scal}[2]{{\langle{{#1}\mid{#2}}\rangle}}
\newcommand{\pair}[2]{{\left\langle{{#1},{#2}}\right\rangle}}
\newcommand{\abs}[1]{{\lvert {#1}\rvert}}
\newcommand{\norm}[1]{{\lVert {#1}\rVert}}
\newcommand{\menge}[2]{\big\{{#1}~\big |~{#2}\big\}}

\newcommand{\HH}{\ensuremath{{\mathcal H}}}

\newcommand{\hh}{\ensuremath{{\EuScript{H}}}}
\newcommand{\C}{\ensuremath{{\mathcal C}}}

\newcommand{\fmap}{\ensuremath{{\Phi}}}
\newcommand{\objfunc}{\ensuremath{\widehat{F}}}

\newcommand{\dictfunc}{\ensuremath{\phi}}
\newcommand{\proj}{\ensuremath{\operatorname{proj}}}

\newcommand{\XC}{\ensuremath{{\mathcal X}}}

\newcommand{\YC}{\ensuremath{{\mathcal Y}}}

\newcommand{\YY}{\ensuremath{\RR}}

\newcommand{\ZZZ}{\ensuremath{{Z}}}
\newcommand{\ud}{\ensuremath{{\mathrm d}}}

\newcommand{\Id}{\ensuremath{\operatorname{Id}}\,}
\newcommand{\cart}{\ensuremath{\raisebox{-0.5mm}{\mbox{\LARGE{$\times$}}}}}

\newcommand{\RR}{\ensuremath{\mathbb{R}}}
\newcommand{\RP}{\ensuremath{{\mathbb R}_+}}
\newcommand{\RPP}{\ensuremath{{\mathbb R}_{++}}}

\newcommand{\soft}[1]{\ensuremath{{\:\operatorname{soft}}_{{#1}}\,}}

\newcommand{\RX}{\ensuremath{\left]-\infty,+\infty\right]}}

\newcommand{\EE}{\ensuremath{\mathsf E}}
\newcommand{\PO}{\ensuremath{\mathsf P}}
\newcommand{\PP}{\ensuremath{ P}}

\newcommand{\QQ}{\ensuremath{ P}}

\newcommand{\NN}{\ensuremath{\mathbb N}}
\newcommand{\KK}{\ensuremath{\mathbb K}}

\newcommand{\pinf}{\ensuremath{{+\infty}}}
\newcommand{\minf}{\ensuremath{{-\infty}}}
\newcommand{\dom}{\ensuremath{\operatorname{dom}}}
\newcommand{\range}{\ensuremath{\operatorname{ran}}}
\newcommand{\prox}{\ensuremath{\operatorname{prox}}}

\newcommand{\sign}{\ensuremath{\operatorname{sign}}}

\begin{document}

\title{Consistent Learning by Composite Proximal Thresholding
\thanks{The work of P. L. Combettes was supported by the
CNRS MASTODONS project under grant 2013MesureHD and by the
CNRS Imag'in project under grant 2015OPTIMISME.}}

\author{Patrick L. Combettes$^1$,\; Saverio Salzo$^2$,\; 
and Silvia Villa$^3$\\[3mm]
\small
\small $\!^1$Sorbonne Universit\'es -- UPMC Univ. Paris 06,
UMR 7598, Laboratoire Jacques-Louis Lions\\
\small F-75005 Paris, France\\
\small \email{plc@ljll.math.upmc.fr}
\\[3mm]
\small
\small $\!^2$Universit\`a degli Studi di Genova,
Dipartimento di Matematica,
16146 Genova, Italy\\
\small \email{saverio.salzo@unige.it}
\\[3mm]
\small
\small $\!^3$Massachusetts Institute of Technology 
and Istituto Italiano di Tecnologia\\
\small Laboratory for Computational and Statistical Learning,
Cambridge, MA 02139, USA\\
\small \email{silvia.villa@iit.it}
}

\maketitle

\begin{abstract}
We investigate the modeling and the numerical solution of
machine learning problems with prediction functions which are 
linear combinations of elements of a possibly infinite-dimensional 
dictionary. We propose a novel flexible composite regularization 
model, which makes it possible to incorporate various priors on 
the coefficients of the prediction function, including sparsity
and hard constraints. We show that the estimators obtained by 
minimizing the regularized empirical risk 
are consistent in a statistical sense, 
and we design an error-tolerant composite 
proximal thresholding algorithm for computing such estimators.
New results on the asymptotic behavior of the proximal 
forward-backward splitting method are derived and exploited to 
establish the convergence properties of the proposed 
algorithm. In particular, our method features a 
$o(1/m)$ convergence rate in objective values.
\end{abstract}

\section{Introduction}
A central task in data science is to extract information from
collected observations. Optimization procedures play a central role
in the modeling and the numerical solution of data-driven
information extraction problems. In the present paper, 
we consider the problem of learning from examples
within the framework of generalized linear models
\cite{BulVan11,Evge00,Gyorfi02}.
The goal is to estimate a functional
relation $f$ from an input set $\XC$ into an output set
$\YC \subset \RR$. The data set consists of 
the observation of a finite number of realizations 
$z_n=(x_i,y_i)_{1\leq i\leq n}$ in $\XC\times\YC$ 
of independent input/ouput random pairs with an unknown common 
distribution $P$. 
We adopt a generalized linear model, i.e., we assume that the
target function $f$ can be approximated by estimators of the form
\begin{equation}
\label{e:model}
f_u=\sum_{k\in\KK}\mu_k\dictfunc_k,
\end{equation}
where $\mathbb{K}$ is at most countable, 
$u=(\mu_k)_{k\in\KK}\in\ell^2(\KK),$ and  
$(\dictfunc_k)_{k\in\KK}$ is a
family of bounded measurable functions from $\XC$ to $\YY$; 
such a family is called a \emph{dictionary}, 
and its elements are called \emph{features}.
The estimator $f_{\widehat{u}_{n,\lambda}}$ is
computed via the approximate minimization of the convex regularized 
empirical risk
\begin{align}
\label{e:problem} 
\widehat{u}_{n,\lambda}&\in\Argmin^{\varepsilon_n}_{u\in\ell^2(\KK)} 
\left(\frac 1 n \sum_{i=1}^n \abs{f_u (x_i)-y_i}^2+\lambda
\sum_{k\in\KK} g_k(\mu_k) \right),
\end{align}
where $\lambda\in\RPP$ and where the convex regularization
functions $(g_k)_{k\in\KK}$ enforce or promote prior knowledge 
on the coefficients $(\mu_k)_{k\in\KK}$ of the decomposition of the
target function $f$ with respect to the dictionary. 
Our objective is to select a family of regularizers 
$(g_k)_{k\in\KK}$ that model a broad range of prior
knowledge and, at the same time, lead to implementable solution 
algorithms that produce consistent estimators as the sample size 
$n$ becomes arbitrarily large. To satisfy this dual objective, we 
shall focus our attention on the following flexible composite
model: each function $g_k\colon\RR\to\RX$ is of the form
\begin{equation}
\label{e:problem1} 
g_k=\iota_{C_k}+\sigma_{D_k}+ h_k,\quad 
h_k-\eta \abs{\cdot}^r\in\Gamma^{+}_0(\RR),
\quad r\in\left]1,2\right],
\quad \eta\in\RPP,
\end{equation} 
where $\iota_{C_k}$ is the indicator function of a closed
interval $C_k\subset\RR$, $\sigma_{D_k}$ is the support 
function of an interval $D_k\subset\RR$, 
$\eta\in\RPP$, and $h_k\colon\RR\to\RP$ is convex and 
such that $h_k(0)=0$. In this model, the role of $C_k$ is
to explicitly enforce hard constraints and the role of $D_k$ is to
promote sparsity \cite{Siop07}. On the other hand,
$h_k$ provides stability and will be seen to be
instrumental in guaranteeing consistency. 
Note that the model \eqref{e:problem}--\eqref{e:problem1}  
refines that considered in \cite{Siop07} and that it
encompasses ridge regression \cite{Gyorfi02,Hoe70}, 
elastic net \cite{Demo09,ZouHastie2005}, bridge regression 
\cite{Fu1998}, and generalized Gaussian models \cite{Anto02}. 
Proximal thresholders \cite{Siop07}, which extend the basic
notion of a soft thresholder, will play a key role in our analysis. 

The main objective of our paper is to investigate statistical and
algorithmic aspects of the estimators based on 
\eqref{e:problem}--\eqref{e:problem1}. Our main contributions are
the following:
\begin{itemize}
\item
We prove the consistency of the estimators 
$(f_{\widehat{u}_{n,\lambda}})_{n\in\NN}$ 
as $n\to\pinf$, as well as the convergence 
of the corresponding coefficients 
$(\widehat{u}_{n,\lambda})_{n\in\NN}$ in $\ell^r(\KK)$. 
This generalizes in particular the
analysis of \cite{Demo09}, which corresponds to the special case
when $C_k=\RR$, $D_k=[-\omega_k,\omega_k]$, and 
$h_k=\eta \abs{\cdot}^2$. In this case, \eqref{e:problem1} reduces
to
\begin{equation}
\label{e:problem9} 
g_k=\omega_k|\cdot|+\eta|\cdot|^2.
\end{equation} 
\item
We establish new asymptotic properties for an error-tolerant 
forward-backward splitting algorithm based on proximal 
thresholders. In particular, we establish new minimizing properties
and a rate of convergence $o(1/m)$ for the objective function values
in the presence of variable proximal parameters, relaxations, 
and computational errors. These results, which are of interest in
their own right, improve on the state of the art, which considers
either the error free-case and the non-relaxed version 
\cite{Bred09,Davi15}, or convergence only in an ergodic 
sense \cite{Schm11}.
\end{itemize}

The paper is organized as follows. In Section~\ref{sec:section},
we set the problem formally and present the main results 
concerning the statistical and algorithmic issues pertaining to
the proposed estimators.
Section~\ref{sec:stat} is devoted to proving the consistency
of the estimators, which is established in Theorem~\ref{thm:main}. 
In Section~\ref{sec:algorithm}, we establish Theorem~\ref{t:222}, 
which concerns the asymptotic behavior of a proximal 
forward-backward splitting algorithm, and 
Theorem~\ref{thm:algoconvergence}, 
which specifically deals with the structure considered in 
\eqref{e:problem}--\eqref{e:problem1}.
Additional properties of the regularizers defined 
in \eqref{e:problem}
are studied in Appendices~\ref{app:lsc} and \ref{app:proxr}.

\noindent
{\bfseries Notation.}
$\NN^*=\NN\smallsetminus\{0\}$, $\RR_+=[0,+\infty[$, and 
$\RPP=\left]0,+\infty\right[$.
Throughout, $\KK$ is an at
most countably infinite index set.
We denote by $(e_k)_{k\in\KK}$ the canonical orthonormal 
basis of $\ell^2(\KK)$. The canonical norm of $\ell^r(\KK)$ is
denoted by $\norm{\cdot}_r$. Let $\HH$ be a real Hilbert space.
We denote by $\scal{\cdot}{\cdot}$ and $\norm{\cdot}$ 
the scalar product and the associated norm of $\HH$.
The set of proper lower semicontinuous convex functions from
$\HH$ to $\RX$ is denoted by $\Gamma_0(\HH)$, and the subset
of $\Gamma_0(\HH)$ of functions valued in $[0,\pinf]$ by
$\Gamma^+_0(\HH)$.
Let $\varphi\in\Gamma_0(\HH)$.
The subdifferential of $\varphi$ at $u\in\HH$ is
$\partial \varphi (u)=\menge{u^*\in\HH}{(\forall v\in\HH)\;
\varphi(u)+\scal{v-u}{u^*}\leq\varphi(v)}$ and,
for every $\varepsilon\in\RPP$, 
$\Argmin_\HH^\varepsilon \varphi 
=\menge{u\in\HH}{\varphi(u)\leq\inf\varphi(\HH)+\varepsilon}$.
Let $\mathcal{D} \subset \HH$.
The indicator function of $\mathcal{D}$ is denoted 
by $\iota_{\mathcal{D}}$ and
the support function of $\mathcal{D}$ is 
$\sigma_\mathcal{D}\colon \HH \to \RX\colon u \mapsto 
\sup_{v\in\mathcal{D}}\scal{v}{u}$.
Let $u\in\HH$. Then 
$\prox_{\varphi}u=\argmin_{v\in\HH}(\varphi(v)+(1/2)\norm{u-v}^2)$ 
\cite{Mor62b}.
Suppose that $\mathcal{D}$ is a nonempty, closed, and convex subset
of $\HH$. Then
$\prox_{\iota_\mathcal{D}}=\proj_{\mathcal{D}}$ is
the projection operator onto $\mathcal{D}$, and 
$\prox_{\sigma_\mathcal{D}}=\Id-\proj_{\mathcal{D}}
=\soft{\mathcal{D}}{}$ is
the soft-thresholder with respect to $\mathcal{D}$.
For background on convex analysis and optimization, see
\cite{Livre1}.

\section{Problem setting and main results}
\label{sec:section}

The following assumption will be made in our main results.

\begin{assumption}\label{ass}{\rm
$(\XC,\mathfrak{A}_{\XC})$ is a measurable space, 
$\YC \subset \RR$ is a nonempty bounded interval, and 
$b=\sup_{y\in\YC} \abs{y}$. Moreover, $\PP$ is
a probability measure on
$\XC\times\YC$ with marginal $\PP_\XC$ 
on $\XC$.
The \emph{risk} is
\begin{equation}
\label{e:risk}
R \colon L^2(\PP_\XC) \to \RP\colon
f\mapsto\int_{\XC\times\YC} \abs{f(x)-y}^2\,\ud P(x,y)
\end{equation}
and $(\dictfunc_k)_{k\in\KK}$ is a family of measurable functions 
from $\XC$ to $\RR$ such that, for some $\kappa\in\RPP$,
\begin{equation}
\label{eq:l2dic}
\sup_{x\in\XC}\sum_{k\in\KK}\abs{\dictfunc_k(x)}^2\leq\kappa^2.
\end{equation}
The \emph{feature map} is
\begin{equation}
\label{eq:featuremap}
\fmap\colon \XC \to \ell^2(\KK)\colon x \mapsto 
(\dictfunc_k(x))_{k\in\KK}
\end{equation}
and
\begin{equation}
\label{eq:fmap_dictionary0}
A\colon \ell^2(\KK) \to \YY^{\XC}\colon u=(\mu_k)_{k\in\KK}
\mapsto f_u=\sum_{k\in\KK} \mu_k \dictfunc_k\;\;
\text{(pointwise)}.
\end{equation}
In addition, 
\begin{enumerate}[label=\rm(\alph*)]
\item
$(C_k)_{k\in\KK}$ is a family of 
closed intervals in $\RR$ such that $0\in\bigcap_{k\in\KK} C_k$.
\item
$(D_k)_{k\in\KK}$ is a family of nonempty
closed bounded intervals in $\RR$ such that 
$\sum_{k\in\KK}|(\inf D_k)_+|^{r^*}<\pinf$
and $\sum_{k\in\KK}|(\inf D_k)_-|^{r^*}<\pinf$.
\item $(h_k)_{k\in\KK}$ is a
family in $\Gamma_0^+(\RR)$ such that
$(\forall k\in\KK)$ $h_k(0)=0$ and 
$h_k -\eta \abs{\cdot}^r\in\Gamma_0^+(\RR)$ 
for some $r\in\left]1,2\right]$ 
and $\eta\in\RPP$.
\end{enumerate}
We define
\begin{equation}
\label{e:problemX} 
\begin{cases}
(\forall k\in\KK)\quad
g_k=\iota_{C_k}+\sigma_{D_k}+h_k\\
F=R\circ A\colon\ell^2(\KK)\to\RR\\
G\colon \ell^2(\KK)\to\RX\colon u \mapsto
\sum_{k\in\KK}g_k(\mu_k)\\
\C=\overline{A\big(\ell^2(\KK) \cap
\cart_{k\in\KK} C_k \big)}\quad\text{(closure is taken 
in $L^2(\PP_\XC)$)}.
\end{cases}
\end{equation} 
$(X_i,Y_i)_{i\in\NN}$ is a sequence of i.i.d.~random variables, on
an underlying probability space
$(\Omega, \mathfrak{A}, \PO)$, taking values in $\XC\times\YC$
and distributed according to $\PP$. For every $n\in\NN^*$, 
$Z_n=(X_i,Y_i)_{1\leq i\leq n}$.
The function $\varepsilon\colon \RPP \to [0,1]$ satisfies 
$\varepsilon(\lambda) \to 0$ as $\lambda \to 0^+$.
Moreover, for every $n\in\NN^*$, 
every $\lambda\in\RPP$, and every training set 
$z_n=(x_i,y_i)_{1\leq i\leq n}\in (\XC\times\YC)^n$
\begin{align}
\label{eq:learnalgo}
\widehat{u}_{n,\lambda}(z_n)&\in
\Argmin^{\varepsilon(\lambda)}_{u\in\ell^2(\KK)} 
\left(\frac 1 n \sum_{i=1}^n \abs{f_u (x_i)-y_i}^2+\lambda
G(u) \right).
\end{align}
}
\end{assumption}

\begin{remark} \
\label{rmk2}
\begin{enumerate}
\item
\label{rmk2i} 
The proposed learning 
method falls into the class of regularized
empirical risk minimization algorithms. However, it differs from
the classical setting which uses the squared
norm as a regularizer \cite{CucSma02,Devito05,Evge00}.
\item
\label{rmk2ii}
The conditions on the sequences 
$((\inf D_k)_+)_{k\in\KK}$ and
$((\sup D_k)_-)_{k\in\KK}$ given in Assumption~\ref{ass}
ensure that $G\in\Gamma_0(\ell^2(\KK))$.
Moreover, $\dom G\subset\ell^r(\KK)$ and $G$ is bounded 
from below and coercive
(see Lemma~\ref{l:20151126a}).
\item 
It follows from \eqref{eq:l2dic} that the linear operator $A$ is
well defined and continuous with respect to the topology of the
pointwise convergence on $\YY^{\XC}$,
that $\range A\subset L^\infty(\PP_{\XC})$, and that 
$A\colon \ell^2(\KK) \to L^2(\PP_{\XC})$ is a bounded linear
operator such that $\norm{A}\leq\kappa$.
The feature map $\Phi$ and $A$ are connected via the identities 
\begin{equation}
\label{e:2015-02-14a}
(\forall k\in\KK)(\forall x\in\XC)\quad
\scal{\fmap(x)}{e_k}=(A e_k) (x).
\end{equation}
In \cite[Proposition 3]{Demo09} it is
shown that $\range A$ can be endowed with a reproducing kernel
Hilbert space structure for which $A$ becomes a partial isometry,
and the corresponding reproducing kernel is
\begin{equation}
K\colon \XC \times \XC \to \RR\colon (x,x^\prime)\mapsto 
\sum_{k\in\KK} \dictfunc_k(x) \dictfunc_k(x^\prime).
\end{equation}
\end{enumerate}
\end{remark}

In the above setting, the goal is to minimize the risk $R$ 
of \eqref{e:risk} on the
closed convex subset $\C$ of $L^2(\PP_\XC)$ using the  $n$
i.i.d.~observations $Z_n=(X_i,Y_i)_{1\leq i\leq n}$.
In this respect, recall that the regression function $f^\dag$ 
is the minimizer of the risk on $L^2(\PP_\XC)$  and that
\begin{equation}
(\forall\, f\in L^2(\PP_\XC))\quad 
R(f)-\inf\, R(L^2(\PP_\XC))=\norm{f-f^\dag}^2_{L^2}.
\end{equation}
This means that minimizing $R$ 
on $L^2(\PP_\XC)$ is equivalent to
approximating the regression function $f^\dag$.
In our constrained setting, the solution to the regression problem
on $\C$ results in a target function $f_\C$ with the following
properties.  

\begin{proposition}
\label{p:regfunc} 
Suppose that Assumption~\ref{ass} is in force.
Then there exists a unique $f_\C\in\C$ such that $R(f_\C) =
\inf R(\C)$. Moreover, the following hold:
\begin{enumerate}
\item\label{p:regfunci} 
$f_\C$ is the projection of $f^\dag$ onto $\C$ in
$L^2(\PP_\XC)$.
\item
\label{p:regfuncii} 
$(\forall f\in\C)$ $\norm{f-f_\C}^2_{L^2}\leq R(f)-\inf R(\C)$.
\item
\label{p:regfunciii}  
$(\forall f\in\C)\;$ $R(f)-\inf R({\C})
\leq \!2\big[\big( \norm{f-f_\C}_{L^2}+\sqrt{ \inf
R(\C)-\inf R(L^2(\PP_\XC))} \big)^2$\\
$\hspace{6.7cm}+\inf R(L^2(\PP_\XC))\big]^{1/2}\norm{f-f_\C}_{L^2}.$
\end{enumerate}
\end{proposition}

Proposition~\ref{p:regfunc} states that, as in the unconstrained
case, minimizing the risk over
$\C$ is still equivalent to approaching $f_\C$ in
$L^2(\PP_\XC)$.  
It is worth noting that we do not assume that
$f_\C=f_u$ for some $u\in\dom G$, since the infimum of $R$ on
$A(\dom G)$ may not be attained. A \emph{consistent learning
scheme} generates a random variable $\widehat{u}_{n,\lambda_n}(Z_n)$,
taking values in $\ell^2(\KK)$,
from $n$ i.i.d.~observations $Z_n=(X_i,Y_i)_{1\leq i\leq n}$, 
so that the resulting sequence of random functions 
$(\hat{f}_n)_{n\in\NN}=(A\widehat{u}_{n,\lambda_n}(Z_n))_{n\in\NN}$ 
is \emph{weakly consistent} in the sense that
\begin{equation}
R(\hat{f}_n) \to\inf R(\C)\;\;\text{in probability} \quad
\Leftrightarrow\quad\norm{\hat{f}_n-f_\C}_{L^2}\to 0\;\text{in
probability},
\end{equation}
or \emph{strongly consistent} in the sense that
\begin{equation}
R(\hat{f}_n) \to\inf R(\C)\quad\PO\text{-a.s.} \quad
\Leftrightarrow\quad\norm{\hat{f}_n-f_\C}_{L^2}\to 0
\quad\PO\text{-a.s.},
\end{equation}
depending on the assumption on the regularization parameters
$(\lambda_n)_{n\in\NN}$.

Next, we first state our consistency result and then
present an algorithm to compute the proposed estimators.

\begin{theorem}
\label{thm:main}
Suppose that Assumption~\ref{ass} is in force and let $f_{\C}$ be
defined as in Proposition~\ref{p:regfunc}.
Let $(\lambda_n)_{n\in\NN}$ be a sequence in 
$\left]0,+\infty\right[$ converging to $0$ and, for every $n\in\NN$,
let $\hat{f}_n=A \widehat{u}_{n,\lambda_n}(Z_n)$.
Then the following hold:
\begin{enumerate}
\item
\label{thm:maini} 
Suppose that $\varepsilon(\lambda_n)/{\lambda_n^{4/r}}\to 0$ 
and that  $1/(\lambda_n^{2/r} n^{1/2}) \to 0$.
Then $(\hat{f}_n)_{n\in\NN}$ is weakly consistent, i.e.,
$\norm{\hat{f}_n-f_\C}_{L^2}\to 0$ in probability.
\item
\label{thm:mainii} 
Suppose that $\varepsilon(\lambda_n)=O(1/n)$ and
that $(\log n)/(\lambda_n^{2/r} n^{1/2}) \to 0$.
Then $(\hat{f}_n)_{n\in\NN}$ is strongly consistent, i.e.,
$\norm{\hat{f}_n-f_\C}_{L^2}\to 0\;\PO\text{-a.s.}$
\item
\label{thm:mainiii} 
Suppose that $f_\C\in A(\dom G)$ and set
$S=\Argmin_{\dom G} F$. Then there exists a unique 
$u^\dag\in S$ which minimizes $G$ over $S$ and $A u^\dag=f_\C$. 
Moreover, the following hold:
\begin{enumerate}
\item
\label{thm:mainiiia}  
Suppose that $\varepsilon(\lambda_n)/\lambda_n^2 \to 0$ and that 
$1/(\lambda_n n^{1/2})\to 0$. Then 
\begin{equation}
\norm{\widehat{u}_{n,\lambda_n}(Z_n)-u^\dag}_r\to 0
\quad\text{in probability}.
\end{equation}
\item  
\label{thm:mainiiib}
Suppose that $\varepsilon(\lambda_n)=O(1/n)$ and that 
$(\log n)/(\lambda_n n^{1/2})\to 0$. Then
\begin{equation}
\norm{\widehat{u}_{n,\lambda_n}(Z_n)-u^\dag}_r \to
0\quad\PO\text{-a.s.}
\end{equation}
\end{enumerate}
\end{enumerate}
\end{theorem}

\begin{remark}{\rm \ 
\begin{enumerate}
\item 
In Theorem~\ref{thm:main}\ref{thm:maini}-\ref{thm:mainii}  the 
weakest conditions on the regularization parameters 
$(\lambda_n)_{n\in\NN}$ occur when $r=2$.
In the case considered in \ref{thm:mainiii}, the
consistency conditions do not depend on the exponent $r$.
\item 
In the special case when, in \eqref{e:problem1}, 
for every $k\in\KK$, $h_k=\eta \abs{\cdot}^r$, $C_k=\RR$, $D_k=[- \omega_k,\omega_k]$,
for some $\omega_k\in\RPP$, 
we recover the elastic net framework of \cite{Demo09} and the same
consistency conditions as in \cite[Theorem 2 and Theorem 3]{Demo09}.
This special case yields a strongly convex problem. 
In our general setting, the exponent $r$ may take any value in $]1,2]$ 
and the objective function is only totally convex on bounded sets 
(see Lemma~\ref{lem:totconvG}).
Note also that our framework allows for the enforcement of hard
constraints. 
\item
Under the hypotheses of \ref{thm:mainiii}, the consistency
extends to the sequence of coefficients 
$(\widehat{u}_{n,\lambda_n}(Z_n))_{n\in\NN}$.
This is relevant when one requires the estimators to mimick the
properties of $u^\dag$.
\item 
When $\KK$ is finite and, for every $k\in\KK$, $g_k=\abs{\cdot}^r$, 
\cite{Kol2009} provides an excess risk bound depending on the
cardinality of $\KK$ and the level of sparsity of $u^\dag$
(see also \cite{Fu1998}). The case $r=1$ has been considered in 
\cite{DauDefDem04}.
Appendix~\ref{app:proxr} collects useful properties of the proximity
operators of power functions.
\end{enumerate}
}
\end{remark}

We now address the algorithmic aspects.
The objective function in \eqref{eq:learnalgo} consists of a smooth
(quadratic) data fitting term and a separable nondifferentiable
term, penalizing each dictionary coefficient individually.  Thus a
natural choice is to consider the forward-backward splitting
algorithm \cite{Smms05}. We stress that, since
$\varepsilon$-minimizers are employed in \eqref{eq:learnalgo},
algorithms that provide minimizing sequences are necessary.
However, when convergence in objective function values is in order,
the current theory is not completely satisfying. 
Indeed, the available results consider only the error free-case
and the unrelaxed version \cite{Bred09,Davi15}.
In \cite{Schm11}, errors are considered, 
but only ergodic convergence is proved.
In the Theorem~\ref{t:222} below, we fill this gap by proving
an $o(1/m)$ rate of convergence in objective values 
with relaxation and in the presence of the following type of errors. 

\begin{definition} 
\label{d:287s54}
Let $\HH$ be a real Hilbert space,
let $\varphi\in\Gamma_0(\HH)$, let $(u,w)\in\HH^2$,
and let $\delta\in\RP$. The
notation $u\simeq_\delta\prox_{\varphi}w$ means that
\begin{equation}
\varphi(u)+\frac{1}{2} \norm{u-w}_\HH^2 \leq\min_{v
\in\HH} \Big( \varphi(v)+\frac{1}{2} 
\norm{v-w}_\HH^2 \Big)+\frac{\delta^2}{2}.
\end{equation}
\end{definition}

\begin{theorem}
\label{t:222}
Let $\HH$ be a real Hilbert space, let $F\colon\HH\to\RR$ be a
convex function which is differentiable on $\HH$ with a 
$\beta$-Lipschitz continuous gradient for some $\beta\in\RPP$. 
Let $G\in\Gamma_0(\HH)$, set $J=F+G$, and
suppose that $\Argmin J \neq \varnothing$.
Let $(\gamma_m)_{m\in\NN}$ be a sequence in $\RPP$ such that
$0 <\inf_{m\in\NN} \gamma_m\leq\sup_{m\in\NN} \gamma_m<2/\beta$, 
let $(\tau_m)_{m\in\NN}$ be a sequence in $\left]0,1\right]$, such 
that $\inf_{m\in\NN} \tau_m>0$. 
Let $(\delta_m)_{m\in\NN}$ be a summable sequence in $\RP$ and let
$(b_m)_{m\in\NN}$ be a summable sequence in $\HH$.
Fix $u_0\in\HH$ and set
\begin{equation}
\label{e:main2}
\begin{array}{l}
\text{for}\;m=0,1,\ldots\\
\left\lfloor
\begin{array}{l}
v_m\simeq_{\delta_m} \prox_{\gamma_m G}
\big(u_m -\gamma_m(\nabla F(u_m)+b_m)\big)\\
u_{m+1}=u_m+\tau_m(v_m-u_m).
\end{array}
\right.\\
\end{array}
\end{equation}
Then the following hold:
\begin{enumerate}
\item
\label{t:1i-1}
$(u_m)_{m\in\NN}$ converges weakly to a point in $\Argmin J$.
\item
\label{t:1i-1Z}
For every $u\in\Argmin J$, 
$\sum_{m\in\NN}\big\|\nabla F(u_m)-\nabla F(u)\|^2<\pinf$.
\item
\label{t:1i}
$\sum_{m\in\NN}\big\|v_m-u_m\big\|^2<\pinf$. 
\item
\label{t:1ii}
$J(u_m)\to\inf J(\HH)$ and
$\sum_{m\in\NN}|J(v_m)-\inf J(\HH)|^2<\pinf$.
\item 
\label{t:1ibis}
Suppose that $\sum_{m\in\NN}(1-\tau_m)<\pinf$.
Then 
\begin{equation}
\nonumber
\sum_{m\in\NN}{\big(J(v_m)-\inf J(\HH)\big)}<\pinf
\quad\text{and}\quad
\sum_{m\in\NN}\big(J(u_m)-\inf J(\HH)\big)<\pinf.
\end{equation}
\item 
\label{t:1iibis}
Suppose that 
$\sum_{m\in\NN}(1-\tau_m)<\pinf$,
$\sum_{m\in\NN} m \delta_m<\pinf$,
and $\sum_{m\in\NN}m \|b_m\|<\pinf$.
Then $J(u_m)-\inf J(\HH)=o(1/m)$.
\end{enumerate}
\end{theorem}

\begin{remark}
In \cite{Bred09}, the rate $O(1/m)$ for objective values 
is proved in the error-free case and no relaxations 
($\delta_m\equiv 0$ and $\tau_m\equiv 1$), 
assuming that $F+G$ is coercive.
On the other hand, an $o(1/m)$ rate on the objective values 
was derived in
\cite{Davi15} in the special case of fixed proximal parameter
$\gamma\in\left]0,2/\beta\right[$, no relaxation, and no errors.
\end{remark}

We now propose the following inexact forward-backward algorithm 
to solve problem~\ref{e:problem}.

\begin{Algorithm}
\label{mainalgo}
Let $(\gamma_m)_{m\in\NN}$ be a sequence in $\RPP$
such that $0<\inf_{m\in\NN}\gamma_m\leq\sup_{m\in\NN}
\gamma_m<\lambda/\kappa^2$, let $(\tau_m)_{m\in\NN}$ be 
a sequence in $\left]0,1\right]$ such that 
$\inf_{m \in\NN} \tau_m>0$.
Let $(b_m)_{m\in\NN}=((\beta_{m,k})_{k\in\KK})_{m\in\NN}
\in(\ell^2(\KK))^{\NN}$ be such that $\sum_{m\in\NN}\|b_m\|<\pinf$,
let $\zeta\in\RPP$,
let $p \in\left]1,+\infty\right[$, and let
$(\xi_k)_{k \in\KK} \in\ell^1(\KK)$.
Fix $(\mu_{0,k})_{k\in\KK}\in\ell^2(\KK)$ and iterate
\begin{equation}
\label{eq:genaccalgo}
\begin{array}{l}
\text{for}\;m=0,1,\ldots\\
\hspace{-1mm}\left\lfloor
\begin{array}{l}
\text{for every}\;k\in\KK\\
\hspace{-1mm}\vspace{3mm}\left\lfloor
\begin{array}{l}
\chi_{m,k}=
\mu_{m,k}-\displaystyle \frac{\gamma_m}{\lambda}
\bigg(\frac{2}{n}\sum_{i=1}^n\Big(\sum_{j\in\KK}
\mu_{m,j}
\dictfunc_{j}(x_i)-y_i\Big)\dictfunc_k(x_i)+
\beta_{m,k}\bigg)\\[6mm]
\abs{\alpha_{m,k}}\leq \displaystyle\frac{\zeta m^{-2p} \xi_k}{4
\gamma_m \max\{h_k (\abs{\chi_{m,k}}+2), 
h_k(-\abs{\chi_{m,k}}-2))\}
+2\abs{\chi_{m,k}}+1}\\[6mm]
\pi_{m,k}=\prox_{\gamma_m h_k}
\big(\soft{\gamma_m D_k} \chi_{m,k} \big)+ \alpha_{m,k}\\[4mm]
\nu_{m,k}=\proj_{C_k}\big(\sign (\chi_{m,k})
\max\big\{0,\sign(\chi_{m,k}) \pi_{m,k}\big\}\big)\\[4mm]
\mu_{m+1,k}=\mu_{m,k}+\tau_m(\nu_{m,k}-
\mu_{m,k}).\\
\end{array}
\right.\\
\end{array}
\right.\\
\end{array}
\end{equation}
\end{Algorithm}

An attractive feature of Algorithm~\ref{mainalgo} is that, at 
each iteration, each component of the functions in 
\eqref{e:problemX} is activated componentwise and individually.

\begin{remark}
Nesterov-like \cite{Nest13} variants of the forward-backward 
splitting algorithm may also be suitable for computing the 
estimators \eqref{eq:learnalgo} to the extent that they also 
generate minimizing sequences \cite{Schm11,Villa13}. 
\end{remark}

\begin{theorem}
\label{thm:algoconvergence} 
Suppose that Assumption~\ref{ass} is in force. Call
\begin{equation}\label{eq:objfunf}
J\colon \ell^2(\KK) \to \RX\colon u=(\mu_k)_{k\in\KK}
\mapsto \frac 1 n \sum_{i=1}^n 
\abs{f_u (x_i)- y_i}^2+\lambda \sum_{k\in\KK} g_k(\mu_k)
\end{equation}
the objective function in \eqref{eq:learnalgo}, 
and let $(u_m)_{m\in\NN}=((\mu_{m,k})_{k\in\KK})_{m\in\NN}$ and 
$(v_m)_{m\in\NN}=((\nu_{m,k})_{k\in\KK})_{m\in\NN}$
be the sequences generated by 
Algorithm~\ref{mainalgo}. Then the following hold:
\begin{enumerate}
\item \label{thm:algoconvergencei}
$J$ has a unique minimizer $\widehat{u}$, 
and $\widehat{u}\in\ell^r(\KK)$.
\item \label{thm:algoconvergenceii}
$\sum_{m\in\NN}|J(v_m)-\inf J(\HH)|^2<\pinf$,
$J(u_m) \to\inf J(\HH)$, $\norm{v_m-\widehat{u}}_r \to 0$, 
and $\norm{u_m-\widehat{u}}_r \to 0$
as $m \to +\infty$. Moreover
\begin{equation}
\label{eq:20151127c}
\norm{v_m-\widehat{u}}_r=O\Big(\sqrt{ J(v_m)-\inf J(\HH)}\Big) 
\end{equation}
and 
\begin{equation}
\label{eq:ccc}
\norm{u_m-\widehat{u}}_r=O\Big(\sqrt{ J(u_m)-\inf J(\HH)}\Big).
\end{equation}
\item 
\label{thm:algoconvergenceiibis}
Suppose that $\sum_{m\in\NN}(1-\tau_m)<\pinf$.
Then 
\begin{equation}
\nonumber
\sum_{m\in\NN}{\big(J(v_m)-\inf J(\HH)\big)}<\pinf \text{ and }
\sum_{m\in\NN}\big(J(u_m)-\inf J(\HH)\big)<\pinf.
\end{equation}
\item \label{thm:algoconvergenceiii}
Suppose that $p>2$, that 
$\sum_{m\in\NN}(1-\tau_m)<\pinf$ and 
$\sum_{m\in\NN}m \norm{b_m}<\pinf$.
Then 
\begin{equation}
J(u_m)-\inf J(\HH) =o(1/m)\quad\text{and}\quad
\norm{u_m-\widehat{u}}_r=o\big(1/\sqrt{m}\big).
\end{equation}
\end{enumerate}
\end{theorem}

\begin{remark}
\label{rmk:genaccalgo}\ 
\begin{enumerate}
\item
\label{rmk:genaccalgoi}
In Algorithm~\ref{mainalgo}, the computation of
$\prox_{\gamma_m h_k}$ tolerates an error $\alpha_{m,k}$. This is
necessary since, in general, the proximity operator is not
computable explicitly. In such instances,
$\prox_{\gamma_m h_k}$ must be computed
iteratively and the bound on
$\abs{\alpha_{m,k}}$ in Algorithm~\ref{mainalgo} gives an
explicit stopping rule for the iterations.
\item
The soft-thresholding operator with respect to a bounded interval 
$D_k=[\underline{\omega}_k,\overline{\omega}_k]\subset \RR$ is
\begin{equation}
(\forall\mu\in D_k)\quad \soft{D_k}{\mu}=
\begin{cases}
\mu-\overline{\omega}_k &\text{if } \mu>\overline{\omega}_k\\
0 &\text{if } \mu\in D_k\\
\mu-\underline{\omega}_k &\text{if } \mu<\underline{\omega}_k.
\end{cases}
\end{equation}
The freedom in the choice of the intervals $(D_k)_{k\in\KK}$,
$(C_k)_{k\in\KK}$, and of the exponent $r$ provides flexibility in 
setting the type of thresholding operation. 
It is in particular possible to promote selective sparsity. 
For instance, taking  $0=\underline{\omega}_k<\overline{\omega}_k$ 
only the positive coefficients are thresholded.
Figures~\ref{fig:fixed_results1} and
\ref{fig:fixed_results2} show a few examples.
\end{enumerate}
\end{remark}

\begin{figure}[t]
\begin{center}
\scalebox{0.9} 
{
\begin{pspicture}(-5,-3)(5,4.2) 
\def\a{0.9}
\def\p{((x-1)^2+256/729*\a^3)^(1/2)}
\def\r{((x+1)^2+256/729*\a^3)^(1/2)}
\psline[linewidth=0.04cm,arrowsize=0.2cm,linestyle=solid]{->}%
(-4,0)(4,0)
\psline[linewidth=0.04cm,arrowsize=0.2cm,linestyle=solid]{->}%
(0,-3.5)(0,3.5)
\psplot[plotpoints=800,algebraic,linewidth=0.03cm,linecolor=red]%
{-4.0}{-1.0}{(x+1)/(1+2*\a)}
\psplot[plotpoints=800,algebraic,linewidth=0.04cm,linecolor=olive]%
{-1.0}{1.0}{0}
\psplot[plotpoints=800,algebraic,linewidth=0.03cm,linecolor=red]%
{1.0}{4.0}{(x-1)/(1+2*\a)}
\psplot[plotpoints=800,algebraic,linewidth=0.03cm,linecolor=orange]%
{-4.0}{-1.0}{(x+1)-9/8*(\a^2)*(1-(1-16*(x+1)/(9*(\a^2)))^(1/2))}
\psplot[plotpoints=800,algebraic,linewidth=0.04cm,linecolor=olive]%
{-1.0}{1.0}{0}
\psplot[plotpoints=800,algebraic,linewidth=0.03cm,linecolor=orange]%
{1.0}{4.0}{(x-1)+9/8*(\a^2)*(1-(1+16*(x-1)/(9*(\a^2)))^(1/2))}
\psplot[plotpoints=800,algebraic,linewidth=0.03cm,linecolor=blue]%
{-4.0}{-1.0}{x+1+(4*\a/(3*2^(1/3)))*((\r-x-1)^(1/3)-(\r+x+1)^(1/3))}
\psplot[plotpoints=800,algebraic,linewidth=0.04cm,linecolor=olive]%
{-1.0}{1.0}{0}
\psplot[plotpoints=800,algebraic,linewidth=0.03cm,linecolor=blue]%
{1.0}{4.0}{x-1+(4*\a/(3*2^(1/3)))*((\p-x+1)^(1/3)-(\p+x-1)^(1/3))}
\psplot[plotpoints=800,algebraic,linewidth=0.03cm,linecolor=olive]%
{-4.0}{-1.0}{x+1}
\psplot[plotpoints=800,algebraic,linewidth=0.04cm,linecolor=olive]%
{-1.0}{1.0}{0}
\psplot[plotpoints=800,algebraic,linewidth=0.03cm,linecolor=olive]%
{1.0}{4.0}{x-1}
\rput(4.2,0){\small${\xi}$}
\rput(0,3.8){\small${{\mathrm{prox}}_g\,\xi}$}
\rput(1.0,-0.4){\small${1}$}
\rput(2.0,-0.0){\small${|}$}
\rput(-2.0,-0.0){\small${|}$}
\rput(-1.0,-0.0){\small${|}$}
\rput(1.0,-0.0){\small${|}$}
\rput(0.0,1.0){\small${-}$}
\rput(0.3,1.0){\small${1}$}
\rput(-0.4,-1.0){\small${-1}$}
\rput(0.0,-1.0){\small${-}$}
\rput(-2.0,0.4){\small${-2}$}
\rput(2.0,-0.4){\small${2}$}
\rput(-1.0,0.4){\small${-1}$}
\end{pspicture} 
}
\end{center}
\caption{Soft thresholding (green) and $\prox_g$ for
$g=\abs{\cdot}+0.9\abs{\cdot}^r$, with $r=2$ (red), 
$r=3/2$ (orange), $r=4/3$ (blue).}
\label{fig:fixed_results1}
\end{figure}
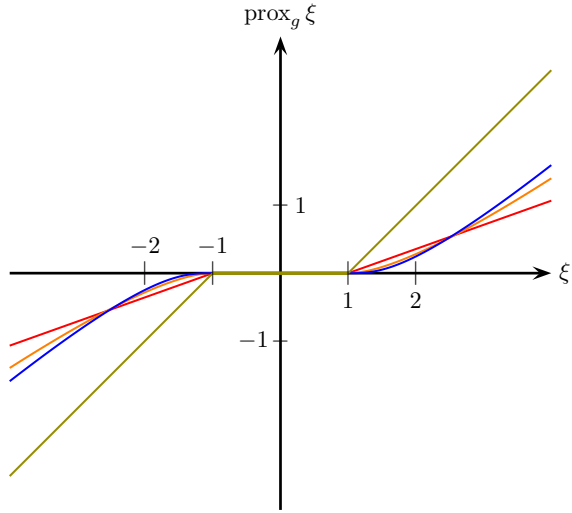

\begin{figure}[t]
\hspace{-2.5cm}
\begin{center}
\scalebox{0.9} 
{
\begin{pspicture}(-4,-2.5)(7,3) 
\def\a{0.9}
\def\r{(x^2+256/729*\a^3)^(1/2)}
\def\p{((x-2)^2+256/729*\a^3)^(1/2)}
\psline[linewidth=0.04cm,arrowsize=0.2cm,linestyle=solid]{->}%
(-3.5,0)(5.5,0)
\psline[linewidth=0.04cm,arrowsize=0.2cm,linestyle=solid]{->}%
(0,-2.5)(0,2.5)
\psplot[plotpoints=800,algebraic,linewidth=0.03cm,linecolor=blue]%
{-3.5}{0.0}{x+(4*\a/(3*2^(1/3)))*((\r-x)^(1/3)-(\r+x)^(1/3))}
\psplot[plotpoints=800,algebraic,linewidth=0.03cm,linecolor=blue]%
{0.0}{2.0}{0}
\psplot[plotpoints=800,algebraic,linewidth=0.03cm,linecolor=blue]%
{2.0}{4.474}{x-2+(4*\a/(3*2^(1/3)))*((\p-x+2)^(1/3)-(\p+x-2)^(1/3))}
\psplot[plotpoints=800,algebraic,linewidth=0.03cm,linecolor=blue]%
{4.471}{5.5}{1.2}

\rput(5.7,0){\small${\xi}$}
\rput(0,2.8){\small${{\mathrm{prox}}_g\,\xi}$}
\rput(2.0,-0.0){\small${|}$}
\rput(4.0,-0.0){\small${|}$}
\rput(-2.0,-0.0){\small${|}$}
\rput(0.0,1.2){\small${-}$}
\rput(0.4,1.2){\small${6/5}$}
\rput(-0.4,-1.0){\small${-1}$}
\rput(0.0,-1.0){\small${-}$}
\rput(-2.0,0.4){\small${-2}$}
\rput(2.0,-0.4){\small${2}$}
\rput(4.0,-0.4){\small${4}$}
\end{pspicture} 
}
\end{center}
\caption{$\prox_g$ for
$g=\iota_{\left]-\infty, 6/5\right]}+\sigma_{[0,2]}+
0.9 \abs{\cdot}^{4/3}$.}
\label{fig:fixed_results2}
\end{figure}
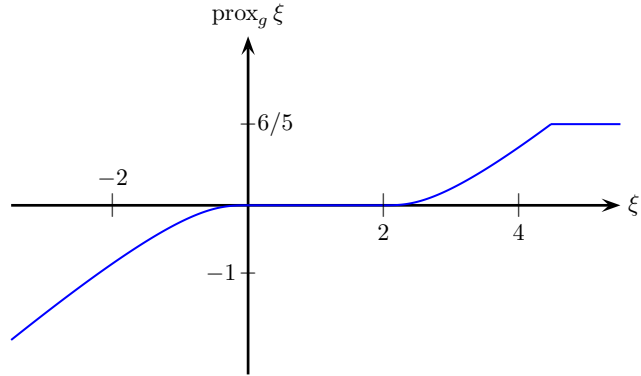

\section{Statistical modeling and analysis}
\label{sec:stat}
Throughout the section Assumption~\ref{ass} is made. 
Our main objective is to prove Theorem~\ref{thm:main}.

The following result establishes that $G$ is totally
convex on bounded sets in $\ell^r(\KK)$ 
and gives an explicit lower bound for the relative 
modulus of total convexity.

\begin{lemma}
\label{lem:totconvG}
Suppose that Assumption~\ref{ass} is in force.
Let $\rho\in\RPP$, let $u_0\in\ell^r(\KK)$ be such that
$\norm{u_0}_r\leq\rho$, let $u_0^*\in\partial G(u_0)$,
and set $M=(7/32)r(r-1)(1- (2/3)^{r-1})$. Then
\begin{equation}
(\forall\, u\in\ell^r(\KK))\quad G(u)-G(u_0) \geq 
\scal{u-u_0}{u_0^*}+ 
\frac{\eta M\norm{u-u_0}_r^2}{(\rho+\norm{u-u_0}_r)^{2-r}}.
\end{equation}
\end{lemma}
\begin{proof}
Let $G_|$ be the restriction of $G$ to $\ell^r(\KK)$, endowed
with the norm $\norm{\cdot}_r$. Since $u_0\in\ell^r(\KK)$ and
$u_0^*\in\ell^{r^*}(\KK)$, we have that $u_0^*\in\partial
G_|(u_0)$. Let $\psi$ be the modulus of total convexity of
$G_|$ and let $\varphi$ be the modulus of total convexity of
$\norm{\cdot}_r^r$ in $\ell^r(\KK)$. Then, for every
$u\in\ell^r(\KK)$, $G(u)-G(u_0) \geq \pair{u-u_0}{u_0^*}+
\psi(u_0;\norm{u-u_0}_r)$.  Moreover, since $G_|=H+\eta
\norm{\cdot}^r_r$, with $H\in\Gamma_0(\ell^r(\KK))$ (see
Lemma~\ref{l:20151126a}), we have $\psi \geq \eta \varphi$. 
The statement follows from 
\cite[Proposition~A.9-Remark~A.10]{arxiv14}.
\end{proof}

The next proposition revisits some results of \cite{Atto96}
about Tikhonov-like regularization specialized to our setting.

\begin{proposition}\label{lem:regularization}
Suppose that Assumption~\ref{ass} is in force.
For every $(\lambda,\epsilon)\in\RPP^2$, 
let $u_{\lambda, \epsilon}$ be an 
$\varepsilon$-minimizer of $F+\lambda G$ and let $u_G$ be the minimizer of $G$.
Then the following hold:
\begin{enumerate}
\item\label{lem:regularization0}
$\inf R(\C)=\inf F(\dom G)$.
\item\label{lem:regularization00}
$(\forall (\lambda,\epsilon) \in\RPP^2)$, 
$\norm{u_{\lambda,\epsilon}-u_G}_r\leq \max\big\{ \norm{u_G}_r,  
\big( 2(F(u_G)+\epsilon)/(\eta M\lambda) \big)^{1/r} \big\}$.
\item
\label{lem:regularizationi}
$F(u_{\lambda,\epsilon}) \to \inf F(\dom G)$ as
$(\lambda,\epsilon) \to (0^+,0^+)$.
\item
\label{lem:regularizationii}
Suppose that 
$S=\Argmin_{\dom G} F\neq\varnothing$ and let
$\varepsilon\colon\RPP \to [0,1]$ be such that 
$\varepsilon(\lambda)/\lambda \to 0^+$ as $\lambda \to 0$.
Then there exists $u^\dag\in\ell^r(\KK)$ such that 
$S=\{u^\dag\}$ and $u_{\lambda, \varepsilon(\lambda)}\to u^\dag$ as
$\lambda\to 0^+$. 
\end{enumerate}
\end{proposition}
\begin{proof}
We first note that it follows from Remark~\ref{rmk2}\ref{rmk2ii}
that $G$ has a minimizer.

\ref{lem:regularization0}: 
Let $u=(\mu_k)_{k\in\KK}\in\ell^2(\KK)\cap  \cart_{k\in\KK} C_k$ 
and take $\delta\in\RPP$.
Then there exists a finite set $\KK_1 \subset \KK$ such that 
$\sum_{k\in\KK\!\smallsetminus\!\KK_1} \abs{\mu_k}^2<\delta^2$.
Now let $v =(\nu_k)_{k\in\KK}$ be such that, for every $k\in\KK_1$,
$\nu_k=\mu_k$
and, for every $k\in\KK\!\smallsetminus\!\KK_1$, $\nu_k=0$. We have
$v\in\dom G$ and $\norm{A u-A v}_{L^2}<\norm{A} \delta$. Thus, 
$\C=\overline{A(\dom G)}$ and the statement follows.

\ref{lem:regularization00}:
Let $(\lambda,\varepsilon) \in\RPP^2$. We derive from the
definition of $u_{\lambda,\varepsilon}$, that
$F(u_{\lambda,\varepsilon})+\lambda G(u_{\lambda,\varepsilon}) 
\leq F(u_G)+\lambda G(u_G)+\varepsilon$
hence, since $0 \in\partial G(u_G)$, it follows from
Lemma~\ref{lem:totconvG} that
\begin{equation}
\frac{\eta M \norm{u_{\lambda,\epsilon}-u_G}_r^2}
{\big(\norm{u_G}_r+\norm{u_{\lambda,\epsilon}-u_G}_r\big)^{2-r}} 
\leq G(u_{\lambda,\epsilon})-G(u_G)\leq
\frac{F(u_G)+\epsilon}{\lambda}.
\end{equation}
If $\norm{u_{\lambda,\epsilon}-u_G}_r \geq \norm{u_G}_r$, then
\begin{equation}
\frac{\eta M \norm{u_{\lambda,\epsilon}-u_G}_r^2}
{\big(\norm{u_G}_r+\norm{u_{\lambda,\epsilon}-u_G}_r\big)^{2-r}} 
\geq  \frac{ \eta M\norm{u_{\lambda,\epsilon}-u_G}_r^2}
{\big(2 \norm{u_{\lambda,\epsilon}-u_G}_r\big)^{2-r}} 
\geq \frac{\eta M }{2} \norm{u_{\lambda,\epsilon}-u_G}_r^r
\end{equation}
and hence $\norm{u_{\lambda,\epsilon}-u_G}^r_r 
\leq 2 (F(u_G) +\epsilon)/ (\eta M \lambda)$.

\ref{lem:regularizationi}: Let $u \in\dom G$.
Then, for every $(\lambda,\epsilon)\in\RPP^2$,
\begin{align}
\nonumber\inf F(\dom G)&\leq F(u_{\lambda,\epsilon})\\
\nonumber&\leq F(u_{\lambda,\epsilon})+\lambda 
\big(G(u_{\lambda,\epsilon})-G(u_G)\big)\\
&\leq F(u)+\lambda \big( G(u)-G(u_G)\big)+\epsilon.
\end{align}
Hence 
\begin{align}
\nonumber\inf F(\dom G) 
&\leq \varliminf_{(\lambda,\epsilon) \to (0,0)}
F(u_{\lambda,\epsilon}) \\
\nonumber&\leq \varlimsup_{(\lambda,\epsilon) \to (0,0)}
F(u_{\lambda,\epsilon})\\
&\leq \varlimsup_{(\lambda,\epsilon) \to (0,0)}\big( F(u)
+\lambda \big( G(u)-G(u_G)\big)+\epsilon\big)\nonumber\\
&\leq F(u).
\end{align}
Since $u$ is arbitrary, the statement follows.

\ref{lem:regularizationii}: 
Since $S$ is convex and $G \in\Gamma_0(\ell^2(\KK))$ is strictly
convex, coercive, and $\dom G \subset \ell^r(\KK)$, it follows 
from \cite[Corollary~11.15(ii)]{Livre1} that there exists 
$u^\dag\in\ell^r(\KK)$ such that $S=\{u^\dag\}$. Moreover, 
\begin{equation}
G(u_{\lambda,\epsilon(\lambda)})
\leq (F(u^\dag)-F(u_{\lambda,\epsilon(\lambda)})
+\epsilon(\lambda))/\lambda+G(u^\dag)
\leq G(u^\dag)+\epsilon(\lambda)/\lambda,
\end{equation}
which implies that
$(G(u_{\lambda,\epsilon(\lambda)}))_{\lambda\in\RPP}$ is bounded. 
Since $G$ is coercive, the family
$(u_{\lambda, \varepsilon(\lambda)})_{\lambda\in\RPP}$ is bounded 
as well. We deduce from 
\cite[Proposition~3.6.5]{Zalinescu02} (see also \cite{ButIusZal03}) 
that there exists an increasing function $\phi\colon\RPP\to\RP$
such that $\phi(0)=0$, for every $t\in\RPP$, $\phi(t)>0$, and
\begin{equation}
(\forall\lambda\in\RPP)\quad\phi
\Big(\frac{\norm{u_{\lambda,\varepsilon(\lambda)}-u^\dag}}{2}\Big)\leq
\frac{G(u^\dag)+G(u_{\lambda,\varepsilon(\lambda)})}{2}-G\Big(
\frac{u_{\lambda,\varepsilon(\lambda)}+u^\dag}{2}\Big).
\end{equation}
Hence, arguing as in \cite[Proof of Proposition~3.1(vi)]{Sico00}, 
we obtain $u_\lambda\to u^\dag$ as $\lambda\to 0^+$.
\end{proof}

Next, we give a representer and stability theorem which generalizes
existing results \cite{DeVito2004,SchHer01} 
to our class of regularization functions.

\begin{theorem}
\label{thm:stability2}
Suppose that Assumption~\ref{ass} is in force. 
Set $M=(7/32)r(r-1)(1- (2/3)^{r-1})$,
let $\lambda\in\RR_{++}$, and let 
$u_\lambda\in\ell^r(\KK)$ be the minimizer of $F+\lambda G$.
Then the following hold:
\begin{enumerate}
\item
\label{thm:stability2i} 
The function
\begin{equation}
\label{eq:Philambda}
\Psi_\lambda\colon \XC\times\YC \to \ell^2(\KK)\colon (x,y) \mapsto
2(f_{u_\lambda}(x)-y) \fmap(x)
\end{equation}
is bounded and $\norm{\Psi_\lambda}_\infty\leq2 \kappa(\kappa
\norm{u_\lambda}_2+b)$.
Moreover
$\norm{\Psi_\lambda}_2\leq2 \kappa \sqrt{R(f_{u_\lambda})}$ and
$-\EE_\PP(\Psi_\lambda) \in\lambda \partial G(u_\lambda)$. 
\item 
\label{thm:stability2ii} 
Let $n \in\NN^*$. Then there exists $\widehat{v}\in\ell^r(\KK)$
such that $\|\widehat{v}
-\widehat{u}_{n,\lambda}(z_n)\|_r\leq \sqrt{\varepsilon(\lambda)}$
\begin{equation}
\frac{\eta M\norm{\widehat{v}-u_\lambda}_r}{(\norm{u_\lambda}_r 
+\norm{\widehat{v}-u_\lambda}_r)^{2-r}}\leq\frac{1}{\lambda}  
\Big(\Big\Vert \frac 1 n \sum_{i=1}^n 
\Psi_\lambda(x_i,y_i)-\EE_{\QQ} (\Psi_{\lambda}) \Big\Vert_{2}
+\sqrt{\varepsilon(\lambda)}\Big).
\end{equation}
\end{enumerate}
\end{theorem}
\begin{proof}
\ref{thm:stability2i}: 
First note that \cite[Corollary~11.15(ii)]{Livre1} asserts that
$u_\lambda$ is well defined, since $F+\lambda G$ is proper and 
lower semicontinuous, and, by Remark~\ref{rmk2}\ref{rmk2ii}, 
strictly convex and coercive. Furthermore 
\cite[Corollary~26.3(vii)]{Livre1} implies that
$-\nabla F(u_\lambda)\in\lambda \partial G(u_\lambda)$. 
We derive from \eqref{eq:featuremap} that 
$A^*\colon L^2(\PP_\XC)\to\ell^2(\KK)\colon f\mapsto
\EE_{\PP_\XC}(f\Phi)$,
 and hence, since $F=R\circ A$, 
 \begin{equation}
(\forall u\in\ell^2(\KK))\quad\nabla F (u)
=A^*\nabla R (f_u)=E_\PP(\varphi),
\end{equation}
where $\varphi\colon(x,y)\mapsto2(f_u(x)-y)\Phi(x)\big)$. 
Let $(x,y)\in\XC\times\YC$. Then
\begin{align}
\abs{f_{u_\lambda}(x)-y}\leq\abs{f_{u_\lambda}(x)}+\abs{y} \leq
\sum_{k \in\KK}\abs{\scal{u_\lambda}{e_k}} \abs{\dictfunc_k(x)}+b
\leq \kappa \norm{u_\lambda}_2+b
\end{align}
and hence
$\norm{\Psi_\lambda(x,y)}_2\leq2 \abs{f_{u_\lambda}(x)-y}
\norm{\fmap(x)}_2\leq 2(\kappa \norm{u_\lambda}_2+b) \kappa$.
Moreover,
\begin{equation}
\int_{\XC\times\YC} \norm{\Psi_\lambda(x,y)}_2^2 \ud P(x,y) \leq
\int_{\XC\times\YC} 
\big(2 \kappa \abs{f_{u_\lambda}(x)-y} \big)^2 \ud P(x,y)=4
\kappa^2 R(f_{u_\lambda}).
\end{equation}
\ref{thm:stability2ii}: 
Let $\widehat{F}_n\colon\ell^2(\KK)\to\RP\colon u 
\mapsto (1/n) \sum_{i=1}^n|f_u(x_i)-y_i|^2$.
Since the restriction of $G$ to $\ell^r(\KK)$ is in 
$\Gamma_0(\ell^r(\KK))$ by Lemma~\ref{l:20151126a}, 
Ekeland's variational principle \cite[Theorem~1.45]{Livre1}
implies that there exists $\widehat{v}\in\ell^r(\KK)$ such that 
$\|\widehat{u}_{n,\lambda}(z_n)-\widehat{v}\|_r\leq
\sqrt{\varepsilon(\lambda)}$  and 
$\inf\|\partial (\widehat{F}_n+\lambda G)(\widehat{v})\|_{r^*}\leq
\sqrt{\varepsilon(\lambda)}$.
Using the inequality $a^2-b^2 \geq 2 (a-b)b$, we derive from
definitions \eqref{eq:Philambda} 
and \eqref{eq:fmap_dictionary0} that, for every 
$i\in\{1,\ldots,n\}$,
\begin{align}
\nonumber \sum_{k \in\KK} \scal{\widehat{v}-u_\lambda}{e_k} 
\scal{\Psi_\lambda(x_i,y_i)}{e_k} &= \sum_{k \in\KK}
\scal{\widehat{v}-u_\lambda}{e_k} 
2 (f_{u_\lambda}(x_i)-y_i) \phi_k(x_i)\\
\nonumber  &= 2 ( f_{\widehat{v}}(x_i)-f_{u_\lambda}(x_i) )
(f_{u_\lambda}(x_i)-y_i) \\
&\leq (y_i- f_{\widehat{v}}(x_i))^2 -(y_i-f_{u_\lambda}(x_i))^2
\end{align}
and, summing over $i$ and dividing by $n$, we obtain
\begin{equation}\label{eq:1}
\widehat{F}_n ( \widehat{v})-\widehat{F}_n(u_\lambda) \geq 
\sum_{k \in\KK} \scal{\widehat{v}-u_\lambda}{e_k} \big\langle \frac
1 n 
\sum\nolimits_{i=1}^n\Psi_\lambda(x_i,y_i)\, \big\vert\, e_k
\big\rangle.
\end{equation}
Lemma~\ref{lem:totconvG} and \ref{thm:stability2i} yield
\begin{equation}\label{eq:2}
\lambda G(\widehat{v})-\lambda G(u_\lambda) \geq 
\scal{\widehat{v}-u_\lambda}{- \EE_\PP (\Psi_\lambda) }+\lambda
\eta M 
\frac{\norm{\widehat{v}-u_\lambda}_r^2}{(\norm{u_\lambda}_r+
\norm{\widehat{v}-u_\lambda}_r)^{2-r}}.
\end{equation}
Next, since $\inf\|\partial (\widehat{F}_n+\lambda G)(\widehat{v})\|_{r^*}\leq \sqrt{\varepsilon(\lambda)}$,
there exists $\widehat{e}^*\in\ell^{r^*}(\KK)$ such that 
$\|\widehat{e}^*\|_{r^*}\leq \sqrt{\varepsilon(\lambda)}$ and
$\scal{u_\lambda-\widehat{v}}{\widehat{e}^*} 
\leq (\widehat{F}_n+\lambda G)(u_\lambda)-(\widehat{F}_n+\lambda G)(\widehat{v})$.
Summing inequalities \eqref{eq:1} and \eqref{eq:2}, we have
\begin{align}
\sqrt{\varepsilon(\lambda)}\|\widehat{v}-u_\lambda\|_r
&\geq (\objfunc_n+\lambda G)(\widehat{v}) 
- (\objfunc_n+\lambda G)(u_\lambda)\nonumber\\
&\geq \sum_{k \in\KK} \scal{\widehat{v}-u_\lambda}{e_k} \Big\langle
\frac 1 n 
\sum_{i=1}^n\Psi_\lambda(x_i,y_i)-\EE_\PP (\Psi_\lambda) \, \Big
\vert \,e_k \Big\rangle \nonumber\\
&\quad\;
+ \frac{\lambda \eta
M\norm{\widehat{v}-u_\lambda}_r^2}{(\norm{u_\lambda}_r 
+ \norm{\widehat{v}-u_\lambda}_r)^{2-r}}.
\end{align}
Hence, using H\"older's inequality,
\begin{equation}
\frac{\lambda \eta M\norm{\widehat{v}-u_\lambda}_r^2}{(\norm{u_\lambda}_r 
+ \norm{\widehat{v}-u_\lambda}_r)^{2-r}} \leq
\norm{\widehat{v}-u_\lambda}_r 
\bigg(\Big\Vert \frac 1 n
\sum\nolimits_{i=1}^n\Psi_\lambda(x_i,y_i)-\EE_\PP (\Psi_\lambda)
\Big\Vert_{r^*}+\sqrt{\varepsilon(\lambda)}\bigg)
\end{equation}
and the statement follows from the fact that $\norm{\cdot}_{r^*}
\leq \norm{\cdot}_{2}$.
\end{proof}

We recall the following concentration inequality in Hilbert spaces
\cite{Livre3}
and give the proof of the main result of this section.

\begin{lemma}[Bernstein's inequality]
\label{thm:bernstein}
Let $(U_i)_{1\leq i\leq n}$ be a finite sequence of i.i.d. random
variables on a probability space 
$(\Omega, \mathfrak{A},\PO)$ and taking values in a real
separable Hilbert space $\hh$. Let $\beta>0$, let $\sigma>0$ and
suppose that
$\max_{1\leq i\leq n} \norm{U_i}\leq\beta$ and that $\EE_\PO
\norm{U_i}^2\leq\sigma^2$.
Then for every $\tau>0$ and every integer $n \geq 1$
\begin{equation}
\PP \bigg[  \bigg\lVert \frac 1 n \sum_{i=1}^n (U_i-\EE_\PP U_i)
\bigg\rVert 
\geq \frac{2 \sigma}{\sqrt{n}}+4 \sigma
\sqrt{\frac{\tau}{n}}+\frac{4 \beta \tau}{3 n} \bigg ] \leq
e^{-\tau}.
\end{equation}
\end{lemma}

\noindent
\begin{proof}[of Proposition~\ref{p:regfunc}]
\ref{p:regfunci}:
For every $f \in\C$, $R(f)=\norm{f-f^\dag}^2_{L^2}+ 
\inf R(L^2(\PP_\XC))$. Therefore, minimizing $R$
over $\C$ turns to find the element of $\C$ which is nearest to
$f^\dag$ in $L^2(\PP_\XC)$.

\ref{p:regfuncii}: It follows from \ref{p:regfunci}, that  
$\inf R(\C)=\norm{f_\C-f^\dag}^2_{L^2}+\inf R(L^2(\PP_\XC))$. 
Therefore, since for every $f \in\C$, 
$\scal{f-f_\C}{f^\dag-f_\C}\leq0$, we have
$R(f)-\inf R(\C)=\norm{f-f^\dag}^2_{L^2}-\norm{f_\C-f^\dag}^2_{L^2}
=\norm{f-f_\C}^2_{L^2}+2\scal{f-f_\C}{f_\C-f^\dag}
\geq\norm{f-f_\C}^2_{L^2}$.

\ref{p:regfunciii}: 
Let $f \in\C$. Using the fact that, 
for every $(a,b,c)\in\RR^3_+$ with 
$a\leq b$, $\sqrt{a+c}-\sqrt{b+c}\leq\sqrt{a}-\sqrt{b}$, 
we derive that
\begin{equation}
\sqrt{R(f)}-\sqrt{\inf R(\C)} \leq
\norm{f-f^\dag}_{L^2}-\norm{f_\C-f^\dag}_{L^2}
\leq \norm{f-f_\C}_{L^2}.
\end{equation}
Therefore, using the inequality $a^2-b^2\leq 2a(a-b)$, we obtain
\begin{align}
R(f)-
&\inf R(\C)\leq2 \sqrt{R(f)}\, \norm{f-f_\C}_{L^2}\nonumber \\
&= 2 \sqrt{\big(\norm{f-f^\dag}^2_{L^2}+ 
\inf R(L^2(\XC))\big)}\,\norm{f-f_\C}_{L^2} \nonumber\\
&\leq 2 \Big( \big( \norm{f-f_\C}_{L^2}+\norm{f_\C-f^\dag}_{L^2}
\big)^2+\inf R(L^2(\XC))\Big)^{1/2}\norm{f-f_\C}_{L^2} \nonumber\\
&=2\Big(\Big(\norm{f-f_\C}_{L^2}+\sqrt{\inf\nolimits_\C
R-\inf\nolimits_{L^2(\PP_\XC)} R}  \Big)^2+ 
\inf R(L^2(\XC))\Big)^{1/2}\norm{f-f_\C}_{L^2}
\end{align}
\end{proof}

\medskip
\noindent
\begin{proof}[of Theorem~\ref{thm:main}]~\ref{thm:maini}: 
Let $n \in\NN^*$, let $z_n= (x_i, y_i)_{1\leq i\leq n} \in (\XC
\times \YC)^n$ and let $\widehat{F}_n\colon u \in\ell^2(\KK) \to
\RP\colon u \mapsto (1/n) \sum_{i=1}^n |f_u(x_i)-y_i|^2$. Let
$u_G\in\Argmin G$, let $\lambda\in\RPP$, and let
$\rho_\lambda=\max\big\{1, \norm{u_G}_r, {( 2(b+\kappa
\norm{u_G}+1)^2/(\eta M \lambda))}^{1/r} \big\}$.
Since $F(u_G)\leq(b+\kappa \norm{u_G})^2$ and
$\objfunc_n(u_G)\leq (b+\kappa \norm{u_G})^2$, from
the definition of $\rho_{\lambda}$ and
Proposition~\ref{lem:regularization}\ref{lem:regularization00} we
derive that $\norm{u_{\lambda}-u_G}_r\leq\rho_{\lambda}$ and
$\norm{\widehat{u}_{n,\lambda}(z_n)-u_G}_r\leq\rho_{\lambda}$ .
It follows from Theorem~\ref{thm:stability2} that there exist
$\Psi_\lambda\colon\XC\times\YC\to \ell^2(\KK)$ and
$\widehat{v}\in\ell^r(\KK)$ such that
$\|\widehat{v}-\widehat{u}_{n,\lambda}(z_n)\|\leq
\sqrt{\varepsilon(\lambda)}$ and 
\begin{equation} 
\frac{M\eta\norm{\widehat{v}-u_\lambda}_r}
{(\norm{u_\lambda}_r+\norm{\widehat{v}-u_\lambda}_r)^{2-r}}
\leq \frac{1}{\lambda}\bigg(  \Big\Vert \frac 1 n \sum_{i=1}^n
\Psi_\lambda(x_i,y_i)-\EE_{\QQ} (\Psi_{\lambda})
\Big\Vert_{2}+\sqrt{\varepsilon(\lambda)}\bigg).
\end{equation}
Therefore
\begin{equation}\label{eq:stab2}
\norm{\widehat{v}-u_\lambda}_r \leq
\frac{(4\rho_\lambda)^{2-r}}{M\eta \lambda}
\left(\big\lVert \EE_\PP (\Psi_\lambda)-\frac{1}{n} \sum_{i=1}^n
\Psi_\lambda(x_i,y_i)\big\rVert_{2}+\sqrt{\varepsilon(\lambda)}
\right).
\end{equation}  
Thus,
\begin{equation}\label{eq:stab2bis}
\norm{\widehat{u}_{n,\lambda}(z_n)-u_\lambda}_r \leq
\sqrt{\varepsilon(\lambda)}+ \frac {(4\rho_\lambda)^{2-r}}{M\eta
\lambda}
\bigg(\big\lVert \EE_\PP(\Psi_\lambda)-\frac{1}{n} \sum_{i=1}^n
\Psi_\lambda(x_i,y_i)\big\rVert_{2}+
\sqrt{\varepsilon(\lambda)}\bigg).
\end{equation}  
Now, consider the i.i.d.~random vectors 
$\Psi_\lambda(X_i,Y_i)\colon\Omega \to \ell^2(\KK)$, for
$1\leq i\leq n$. It follows from
Theorem~\ref{thm:stability2}\ref{thm:stability2i} that 
$\max_{1\leq i\leq n}\norm{\Psi_\lambda(X_i,Y_i)}\leq2
\kappa(\kappa \rho_\lambda+b)$ and that 
$\max_{1\leq i\leq n}\EE_{\PO}\norm{\Psi_\lambda(X_i,Y_i)}^2
\leq 4 \kappa^2 R( f_{u_\lambda})$. 
Now set $\beta_\lambda=2\kappa(\kappa \rho_\lambda+b)$ and 
$\sigma_\lambda^2= \kappa^2 R( f_{u_\lambda})$.
Then Bernstein's inequality in Hilbert spaces 
(Lemma~\ref{thm:bernstein}) gives
\begin{equation}
(\forall\tau\in\RPP)\quad\PO  \Big[ \Big\lVert \EE
(\Psi_\lambda(X,Y))-\frac{1}{n} \sum_{i=1}^n \Psi_\lambda(X_i,
Y_i)\Big\rVert_{2}\leq \delta(n,\lambda, \tau) \Big]  \geq 1
-e^{-\tau},
\end{equation}
where $\delta(n,\lambda, \tau)=2\sigma_\lambda/\sqrt{ n}
+ 4\sigma_\lambda \sqrt{\tau/ n}+4\beta_\lambda \tau /(3 n)$.
Thus, recalling \eqref{eq:stab2bis} we have
\begin{equation}
\label{eq:sampleerr1}
\PO  \Big[  \norm{\widehat{u}_{n,\lambda}(\ZZZ_n)-u_\lambda}_r 
> \sqrt{\varepsilon(\lambda)}+ \frac
{(4\rho_\lambda)^{2-r}}{M\eta\lambda} \big(\delta(n,\lambda, \tau)
+\sqrt{\varepsilon(\lambda)}\big)\Big]
\leq e^{-\tau}.
\end{equation}
Set $\gamma_0=2(b+\kappa \norm{u_G}+1)^2$ and 
$\gamma_1=4^{2-r} \gamma_0^{2/r-1}/ (\eta M)^{2/r}$.
We note that, since $\sigma_\lambda$ is bounded, 
say by $\gamma_2$, for $\lambda<1$
sufficiently small,
we have
\begin{align}
\label{eq:ooo}
\nonumber&\frac {(4\rho_\lambda)^{2-r}}{M\eta \lambda} \big(\delta(n,\lambda,
\tau) +\sqrt{\varepsilon(\lambda)}\big)\\
\nonumber&= \frac {4^{2-r}}{M\eta} \bigg(\frac{\gamma_0}{\eta M}\bigg)^{\frac 2
r-1} \frac{1}{\lambda^{2/r}} 
\bigg( \frac{2\sigma_\lambda}{\sqrt{n}}+4 \sigma_\lambda
\sqrt{\frac \tau n} 
+ \frac{4 \beta_\lambda \tau}{3 n}+\sqrt{\varepsilon(\lambda)} \bigg)\\
&\leq \gamma_1
 \bigg( \frac{2 \gamma_2}{ \lambda^{2/r}n^{1/2}}+4 \gamma_2\frac{\sqrt{\tau}}{\lambda^{2/r}n^{1/2}}
+ \frac{8 \tau \kappa^2 \gamma_0/(\eta M)^{1/r}}{3 n \lambda^{3/r}}
+ \frac{8 \tau \kappa b}{3 n \lambda^{2/r}} 
+ \frac{\sqrt{\varepsilon(\lambda)}}{\lambda^{2/r}}\bigg).
\end{align}
Therefore, since $1/(\lambda_n^{2/r} n^{1/2}) \to 0$ and ${\sqrt{\varepsilon(\lambda_n)}}/{\lambda_n^{2/r}}\to 0$
it follows that 
\begin{equation}
\frac{(4 \rho_{\lambda_n})^{2-r}}{\lambda_n}\Big( \delta(n, \lambda_n,
\tau)+\sqrt{\varepsilon(\lambda_n)}\Big)\to 0
\end{equation}
and hence, in view of \eqref{eq:sampleerr1}, we get
$\norm{\widehat{u}_{n,\lambda_n}(\ZZZ_n)-u_{\lambda_n}}_r \to 0$ in
probability.
Moreover, using Proposition~\ref{p:regfunc}\ref{p:regfuncii},
\begin{align}
\nonumber\norm{\widehat{f}_n-f_\C}_{L^2} &
\leq \norm{A \widehat{u}_{n,\lambda_n}(Z_n)-A u_{\lambda_n}}_{L^2}
+\norm{A u_{\lambda_n}-f_\C}_{L^2}\\
\label{eq:regru}&\leq \norm{A}\norm{\widehat{u}_{n,\lambda_n}(Z_n)-
u_{\lambda_n}}_{r}+\sqrt{F(u_{\lambda_n})-\inf F(\dom G)}.
\end{align}
Since $F(u_{\lambda_n})-\inf F(\dom G) \to 0$ by 
Proposition~\ref{lem:regularization}\ref{lem:regularizationi}, and 
$\norm{ \widehat{u}_{n,\lambda}(Z_n)-u_{\lambda_n}}_{r}\to 0$
in probability, we derive that $\norm{\widehat{f}_n-f_\C}_{L^2}\to 0$
 in probability.
 
\ref{thm:mainii}:
Let $n\in\NN^*$, let $\eta\in\RPP$, and set 
\begin{equation}
\Omega_{n,\eta}=\Big\{\norm{\widehat{f}_n-f_\C}_{L^2}>\norm{A}\eta
+\sqrt{F(u_{\lambda_n})-\inf F(\dom G)}\Big\}.
\end{equation}
Since $\varepsilon(\lambda_n)=O(1/n)$, it follows from
\eqref{eq:ooo} that there exists $\gamma_3\in\RPP$
such that, for every $\tau\in\left]1,+\infty\right[$,  
and every $n \in\NN^*$,
\begin{equation}
\label{eq:ltrm}
\frac{(4\rho_{\lambda_n})^{2-r}}{\eta M\lambda_n} \big(\delta(n,\lambda_n, \tau) 
+\sqrt{\varepsilon(\lambda_n)}\big)
\leq \frac{\gamma_3\tau}{\lambda_n^{2/r}n^{1/2}}.
\end{equation}
Let $\xi\in\left]1,+\infty\right[$. There exists $\overline{n}\in\NN^*$,
such that, for every integer $n\geq \overline{n}$,
\begin{equation}
\label{eq:eta}
\frac{\gamma_3}{\lambda_n^{2/r}n^{1/2}}\leq 
\frac{\gamma_3\xi\log n}{\lambda_n^{2/r}n^{1/2}}\leq\eta.
\end{equation}
Therefore, it follows from \eqref{eq:sampleerr1}, \eqref{eq:regru},
\eqref{eq:ltrm}, and \eqref{eq:eta} that, for $n$ large enough,
\begin{equation}
\PO\Omega_{n,\eta} \leq
\exp\left(-\frac{\eta\lambda_n^{2/r}n^{1/2}}{\gamma_3} \right)
\leq  \exp\left(-\xi\log n\right)=n^{-\xi}.
\end{equation}
Thus, $\sum_{n=\bar{n}}^{+\infty} 
\PO\Omega_{n,\eta}<+\infty$
and we derive from the Borel-Cantelli lemma that 
$\PO\big(\bigcap_{k\geq\bar{n}}\bigcup_{n\geq k}\Omega_{n,\eta}\big)=0$.
Recalling Proposition~\ref{lem:regularization}\ref{lem:regularizationi},
we conclude that the sequence $\|\widehat{f}_n-f_{\mathcal{C}}\|_{L^2}\to 0$
$\PO\verb0-0\text{a.s.}$

\ref{thm:mainiii}:~First note that 
Proposition~\ref{lem:regularization}\ref{lem:regularizationi} 
implies that $u^\dag$ is well defined and
that $\rho=\sup_{\lambda\in\RPP}\norm{u_\lambda}<+\infty$.
Now, let $\lambda \in\RPP$ and let $n\in\NN^*$.
Since $\norm{u_\lambda}\leq \rho$, arguing as in the proof 
of \ref{thm:maini}, we obtain
\begin{equation}
\label{eq:sampleerr2}
(\forall \tau\in\RPP)\;\PO  \Big[
\norm{\widehat{u}_{n,\lambda}(\ZZZ_n)-u_\lambda}_r >
\sqrt{\varepsilon(\lambda)}  +
\frac {(4\rho)^{2-r}}{M\lambda}
\big(\delta(n,\tau)+\sqrt{\varepsilon(\lambda)} \big)\Big] \leq
e^{-\tau},
\end{equation}
where $\sigma=2\kappa(\kappa\rho+b)$ and $\delta(n, \tau)
=4\sigma/\sqrt{ n}+ 4\sigma \sqrt{\tau/ n}+4\sigma \tau /(3 n)$.

\ref{thm:mainiiia}:~Since $1/(\lambda_n n^{1/2}) \to 0$, we have 
$(1/\lambda_n) \delta(n,\tau)\to 0$
and hence in view of \eqref{eq:sampleerr2},
$\norm{\widehat{u}_{n,\lambda_n}(\ZZZ_n)-u_{\lambda_n}}_r \to 0$ in probability.
Moreover, since
$\norm{u_{n,\lambda_n}(\ZZZ_n)-u^\dagger}\leq
\norm{u_{n,\lambda_n}(\ZZZ_n)
-u_{\lambda_n}}+\norm{u_{\lambda_n}-u^\dag}$,
the statement follows by 
Proposition~\ref{lem:regularization}\ref{lem:regularizationii}. 
 
\ref{thm:mainiiib}: The proof follows the same line as that of \ref{thm:mainii}.
\end{proof}

\section{Algorithm}
\label{sec:algorithm}
The goal of this section is to prove Theorem~\ref{t:222} and
Theorem~\ref{thm:algoconvergence}. The proof of 
Theorem~\ref{t:222} is based on the following fact.

\begin{lemma}{\rm\cite[Lemma~4.1]{Villa13}}
\label{l:20151125a}
Let $\HH$ be a real Hilbert space, let $\beta \in\RPP$, and let
$\delta \in\RP$.  Let $F\colon \HH \to \RX$ be a convex
differentiable function with  $\beta$-Lipschitz continuous
gradient, and let $G \in\Gamma_0(\HH)$. Then, for every $(u,v,w)
\in\HH^3$ and every $v^* \in\partial_\delta G(v)$,
\begin{equation}
(F+G)(v)\leq(F+G)(u) + \scal{ v-u}{\nabla F(w)+v^*} 
+ \frac{\beta}{2} \norm{v-w}^2+\delta.
\end{equation}
\end{lemma}

\begin{proof}[of Theorem~\ref{t:222}]
Let $m\in\NN$ and set 
\begin{equation}
\tilde{v}_m=\prox_{\gamma_m G} (u_m -\gamma_m(\nabla F(u_m)+b_m)).
\end{equation}
Since
\begin{equation}\label{eq:prox}
v_m\in\Argmin^{\delta_m^2/(2\gamma_m)}_{w\in\mathcal{H}}
\left\{G(w)+\frac{1}{2\gamma_m}
\|w-(u_m -\gamma_m(\nabla F(u_m)+b_m))\|^2 \right\},
\end{equation}
using the strong convexity of the objective function in
\eqref{eq:prox}, we get 
\begin{equation}
\|v_m-\tilde{v}_m\|\leq \delta_m.
\end{equation}
Therefore, setting $a_m=v_m-\tilde{v}_m$, we have
\begin{equation}
u_{m+1}=u_{m}+\tau_m \big(\prox_{\gamma_m G} (u_m -\gamma_m(\nabla
F(u_m)+b_m))+a_m-u_m\big).
\end{equation}
Hence \eqref{e:main2} is an instance of
the inexact forward-backward algorithm studied in \cite{Smms05}
and we can therefore use the results of 
\cite[Theorem 3.4]{Smms05}.

\ref{t:1i-1}--\ref{t:1i-1Z}:
The statements follow from \cite[Theorem 3.4(i)-(ii)]{Smms05}.

\ref{t:1i}: We have
\begin{align}
\label{eq:sum}
\nonumber\|u_m-v_m\|^2&\leq 2 \|u_m-\tilde{v}_m\|^2+2 \|a_m\|^2\\
&\leq 4 \|u_m-\prox_{\gamma_m G}(u_m-\gamma_m\nabla F(u_m))\|^2
+4\|b_m\|^2+2 \|a_m\|^2.
\end{align}
Therefore, the statement follows from 
\cite[Theorem 3.4(iii)]{Smms05}. 

\ref{t:1ii}:
By \eqref{eq:prox} and \cite[Lemma~1]{SalVil12}, there exist
$\delta_{1,m}\in[0,+\infty[$, $\delta_{2,m}\in[0,+\infty[$, and
$e_m\in\mathcal{H}$ with $\delta_{1,m}^2+\delta_{2,m}^2\leq
\delta_m^2$ and $\|e_m\|\leq \delta_{2,m}$ such that
\begin{equation}
\label{e:20151119d}
v^{*}_m=\frac{u_m-v_m}{\gamma_m}  -(\nabla F(u_m)+b_m)+
\frac{e_m}{\gamma_m}\in\partial_{{\delta_{1,m}^2}/(2\gamma_m)}
G(v_m). 
\end{equation}
Now set $J=F+G$. It follows from Lemma~\ref{l:20151125a} that, for
every $u \in\HH$,
\begin{align}
\label{e:20151123a}
\nonumber J(v_m) &- J(u)\\
\nonumber &\leq \scal{ v_m-u}{\nabla F(u_m)+v^*_m} 
+ \frac{\beta}{2} \norm{v_m-u_m}^2+\frac{\delta^2_{1,m}}{2
\gamma_m}\\
\nonumber &=\frac{1}{\gamma_m}\scal{v_m-u}{ u_m-v_m}  
+ \frac{\beta}{2} \norm{v_m-u_m}^2 \\
\nonumber &\hspace{30ex}+
\frac{1}{\gamma_m}\scal{v_m-u}{e_m-\gamma_m b_m} 
+ \frac{\delta^2_{1,m}}{2 \gamma_m}\\
&=\frac{1}{2 \gamma_m}\big(\norm{u_m-u}^2-\norm{v_m-u}^2 \big)
+ \frac 1 2 \bigg( \beta-\frac{1}{\gamma_m}\bigg) \norm{v_m-u_m}^2\\ 
\nonumber &\hspace{30ex}
+ \frac{1}{\gamma_m}\scal{v_m-u}{e_m-\gamma_m b_m} 
+ \frac{\delta^2_{1,m}}{2 \gamma_m}.
\end{align}
We derive from \ref{t:1i-1} and \ref{t:1i} that 
${(\scal{v_m-u}{u_m-v_m})}_{m \in\NN}$ is square summable.
Therefore, if we let $u \in\Argmin J$, it follows from
\eqref{e:20151123a} that $(J(v_m)-\inf J(\HH))_{m\in\NN}$ is square
summable.
Now, if we let $u=u_m$ in \eqref{e:20151123a} we have
\begin{align}
\label{e:20151123b}
J(v_m)-J(u_m)
&\leq 
\bigg(\frac{\beta}{2}-\frac{1}{\gamma_m} \bigg) 
\norm{u_m-v_m}^2 
+\frac{1}{\gamma_m} \bigg(\norm{u_m-v_m} \norm{e_m-\gamma_m b_m}
+\frac{\delta_{1,m}^2}{2}\bigg)\nonumber\\
&\leq\frac{1}{\gamma_m} \bigg(\norm{u_m-v_m} \norm{e_m-\gamma_m b_m}
+\frac{\delta_{1,m}^2}{2}\bigg).
\end{align}
Set $\underline{\gamma}=\inf_{m \in\NN} \gamma_m$. Since 
$u_{m+1}=u_m+\tau_m (v_m-u_m)$, 
using the convexity of $J$ and \eqref{e:20151123b}, we get
\begin{align}
\label{e:20151123c}
J(u_{m+1})-\inf J(\HH)
&\leq J(u_m)-\inf J(\HH) + \tau_m (J(v_m)-J(u_m))\nonumber\\
&\leq J(u_m)-\inf J(\HH) 
+ \underline{\gamma}^{-1} \big( \|u_m-v_m\| \norm{e_m -\gamma_m
b_m}+\delta^2_{1,m}/2 \big).
\end{align}
Thus, since
${\big( \|u_m-v_m\|\,\norm{e_m -\gamma_m b_m}+\delta^2_{1,m}/2
\big)}_{m \in\NN}$
is summable, \cite[Lemma~2.2.2]{LivrePol},
ensures that $( J(u_m)-\inf J(\HH))_{m \in\mathbb{N}}$ converges.
In view of the inequalities in \eqref{e:20151123c} that
its limit must be $0$.

\ref{t:1ibis}: Let $u \in\Argmin J$. Since, 
$u_m-u=(1-\tau_{m-1}) (u_{m-1}-u)+\tau_{m-1}(v_{m-1}-u)$,
it follows from the convexity of $\norm{\cdot}^2$ that
\begin{equation}
\norm{u_m-u}-\norm{v_m-u}^2\leq(1-\tau_{m-1}) \norm{u_{m-1}-u}^2 
+ \norm{v_{m-1}-u}^2-\norm{v_m-u}^2.
\end{equation}
Therefore, it follows from \eqref{e:20151123a}
that
\begin{align}
\label{e:20151123d}
\nonumber 0 
&\leq J(v_m)-J(u) \\ 
&\leq\nonumber\frac{1-\lambda_{m-1}}{2 \underline{\gamma}} \norm{u_{m-1}-u}^2
+ \frac{1}{2\underline{\gamma}} \big( \norm{v_{m-1}-u}^2-\norm{v_m-u}^2 \big)\\
 &+ \frac 1 2 \bigg( \beta-\frac{1}{\gamma_m}\bigg) 
\norm{v_m-u_m}^2+\frac{1}{\underline{\gamma}} \bigg(\norm{v_m-u} 
\norm{e_m-\gamma_m b_m}
+ \frac{\delta_{1,m}^2}{2}\bigg).
\end{align}
Hence, ${(J(v_m)-\inf J(\HH))}_{m \in\NN}$ is summable, 
for each term on the right hand side of \eqref{e:20151123d} is summable. 
Since $u_{m+1}=(1- \tau_m) u_m+\tau_m v_m$, convexity of $J$ yields
\begin{equation}
0\leq J(u_{m+1})-\inf J(\HH)\leq(1-\tau_m)\big( J(u_m)-\inf J(\HH) \big)
+ \tau_m\big( J(v_m)-\inf J(\HH)\big).
\end{equation}
The summability of $(1-\tau_m)_{m \in\NN}$ and 
${\big( J(v_m)-\inf J(\HH)\big)}_{m \in\NN}$ implies that of
${\big( J(u_m)-\inf J(\HH)\big)}_{m \in\NN}$.

\ref{t:1iibis}:
Since ${\big( J(u_m)-\inf J(\HH)\big)}_{m \in\NN}$ is summable, it follows from
\eqref{e:20151123c} and \cite[Lemma~3]{Davi15}
that $\big( J(u_m)-\inf J(\HH)\big)=o(1/m)$.
\end{proof}

The purpose of the rest of the section is to show how
approximations of the type considered in Theorem~\ref{t:222}
(equation \eqref{e:main2}) can be computed explicitly.

\begin{lemma}\label{l:20150704a}
Let $h\colon\RR\to\RR$ be convex and such that 
$0 \in\Argmin_{\RR} h$, let $(s,\mu) \in\RR^2$, and 
let $\alpha \in [-1,1]$. 
Let $\beta\in\RP$ be the Lipschitz constant of $h$ in 
$[-1-\abs{\mu},1+\abs{\mu}]$ and set
\begin{equation}
\delta=\sqrt{(2 \beta+2 \abs{\mu}+1)\abs{\alpha}}\qquad
\text{and}\qquad
s=\prox_h\mu+\alpha.
\end{equation}
Then $s \simeq_\delta \prox_h\mu$.
Moreover, $\widehat{s}=\sign(\mu) \max\{0, \sign(\mu)s\}$ satisfies 
$\widehat{s} \simeq_\delta \prox_h\mu$
and $\mu\widehat{s}\geq 0$.
\end{lemma}
\begin{proof}
Let $t=\prox_h\mu$.
Since $0 \in\Argmin h$, $\prox_h0=0$. 
Hence, since $\prox_h$ is nonexpansive and increasing
\cite[Lemma~2.4]{Smms05}, $\abs{t}\leq\abs{\mu}$ and 
$\sign(t)=\sign(\mu)$. We note that
$\abs{s}\leq\abs{s-t}+\abs{t}\leq1+\abs{\mu}$. Thus,
\begin{align}
\nonumber h(s)+ \frac 1 2 \abs{s-\mu}^2-h(t)-\frac 1 2
\abs{t-\mu}^2 &\leq \beta \abs{s-t} 
+ \frac 1 2 \abs{s-t} \abs{s- \mu+t- \mu}\\
&\leq\frac 1 2 (2 \beta+1+2\abs{\mu}) \abs{\alpha}.
\end{align}
To conclude, it is enough to note that 
$\abs{\widehat{s}-\prox_{h}(\mu)}\leq\abs{\alpha}$.
\end{proof}

\begin{lemma}
\label{lem:compprox} 
Let $h\in\Gamma_0(\RR)$, let $\sigma\in\Gamma_0(\RR)$ be a support
function, and set $\phi=h+\sigma$. 
Let $(s,x)\in\RR^2$ be such that $s
\prox_\sigma(x)\geq 0$, and let $\delta\in\RP$. Then
\begin{equation}
s\simeq_{\delta}\prox_{h}(\prox_\sigma x)\quad\Rightarrow\quad
s\simeq_{\delta}\prox_{\phi}x.
\end{equation}
\end{lemma}
\begin{proof}
Let $\mu=\prox_\sigma x$ and
$s\simeq_{\delta}\prox_{h}(\prox_\sigma x)$.
By \cite[Lemma 2.4]{SalVil12} there exist
$(\delta_1,\delta_2)\in\RP^2$ and $e\in\RR$, such that
\begin{equation}
\mu-s+e\in\partial_{{\delta_1^2}/2} h(s), \quad \abs{e}\leq
\delta_2, \quad\text{and}\quad\delta_1^2+\delta_2^2\leq \delta^2. 
\end{equation}
Hence
\begin{equation}
x-s+e=x-\mu+\mu-s+e\in\partial\sigma(\mu)+
\partial_{{\delta_1^2}/2}h(s).
\end{equation}
Since $s\mu \geq 0$, there exists $t\in\RP$ such that $\mu=t s$.
Moreover, since $\sigma$ is positively homogeneous, $\partial\sigma
(ts)\subset \partial\sigma(s)$. Therefore 
$x-s+e\in\partial\sigma(s)+\partial_{{\delta_1^2}/2}h(s)\subset
\partial_{{\delta_1^2}/2}\phi(s)$, which implies that
$s\simeq_{\delta}\prox_{\phi}x$ by \cite[Lemma~2.4]{SalVil12}.  
\end{proof}

\begin{remark}{\rm
Let $h\in\Gamma_0(\RR)$, let  $(s,\mu)\in\RR^2$, and let $\delta
\in\RPP$.
Suppose that $0\in\Argmin_{\RR}h$ and that 
$s\simeq_{\delta}\prox_{h}\mu$ with
 $\delta\leq \abs{s}$. Then $s \mu\geq 0$. 
Indeed, since $h(0)=\inf h(\RR)$, we
have
\begin{equation}
h(s)+\frac 1 2 |s-\mu|^2\leq h(0)+\frac 12 \mu^2+\frac 1 2
\delta^2\leq h(s)+\frac 12 \mu^2+\frac 12 \delta^2
\end{equation}
and hence $0\leq (1/2)(s^2-\delta^2)\leq s\mu$.
This shows that Lemma~\ref{lem:compprox}, when $\delta=0$, gives 
$\prox_\phi=\prox_{h} \circ \prox_\sigma$
and consequently generalizes \cite[Proposition~3.6]{Siop07},
relaxing also the condition on the differentiability of $h$ at $0$.
With the help of this result one can compute general thresholders
operators as the proximity operator of 
$\abs{\cdot}+\eta \abs{\cdot}^r$. 
Figure~\ref{fig:fixed_results1} depicts
some instances of these thresholders 
(see also \cite{Siop07}). 
}
\end{remark}

The following lemma is an error-tolerant version of  
\cite[Proposition~12]{Ieee07}.

\begin{lemma}\label{lemma:inexprojprox}
Let $\phi\in\Gamma_0(\RR)$, let $(s,x,p)\in\RR^3$, let
$\delta\in\RP$, and let $C\subset \RR$ be a nonempty closed
interval. Then
\begin{equation}
s\simeq_{\delta}\prox_{\phi}x,\quad\text{and}\quad 
p=\proj_C s \quad\Rightarrow\quad
p\simeq_{\delta}\prox_{\phi+\iota_C}x.
\end{equation}
\end{lemma}
\begin{proof}
Let $g=\phi+(1/2) (\cdot-x)^2$ and let $\epsilon=(\delta^2/2)$.
Since $g$ is convex and $\bar{s}=\prox_{\phi}x$ is its minimum, $g$
is decreasing on $]-\infty, \bar{s}]$ and increasing on $[\bar{s},
+\infty[$. By definition $s$ is a $\epsilon\verb0-0$minimizer of
$g$. The statement is equivalent to the fact that $p$ is a
$\epsilon\verb0-0$minimizer of $g+\iota_C$. If $s \in C$, then $p$
is a fortiori an $\epsilon\verb0-0$minimizer of $g+\iota_C$. We now
consider two cases. First suppose  that $s<\inf C$. If $s<\inf
C\leq\bar{s}$, then $\inf C$ is still an
$\epsilon\verb0-0$minimizer of $g$ and $\inf C \in C$.
Thus $p=\inf C$ is 
an $\epsilon\verb0-0$minimizer of $g+\iota_C$. If either $s
\leq \bar{s}\leq\inf C$ or $\bar{s}\leq s<\inf C$, we have
$p=\proj_C \bar{s}=\inf C$, which is the minimum of $g+\iota_C$,
since $g$ is increasing on $[\bar{s}, +\infty[$. The second case
$\sup C<s$ is treated likewise.
\end{proof}

\begin{proposition}
\label{p:20150912a}
Let $\HH$ be a separable real Hilbert space and let 
$(o_k)_{k \in\KK}$ be an orthonormal basis of $\HH$, where $\KK$
is an at most countable set.
Let $(h_k)_{k \in\KK}$ be a family of convex functions from $\RR$
to $\RR$ such that, for every $k \in\KK$, $h_k \geq h_k(0)=0$. 
Let $(C_k)_{k \in\KK}$ be a family of 
closed intervals in $\RR$ such that $0\in\bigcap_{k\in\KK} C_k$, let
$(D_k)_{k\in\KK}$ be a family of nonempty
closed bounded intervals in $\RR$. Suppose that 
$(h_k^*(-(\inf D_k)_+))_{k \in\KK}$ and
$(h_k^*((\sup D_k)_-))_{k \in\KK}$ are summable,
and set 
\begin{equation}
G\colon\HH\to\RX\colon u 
\mapsto\sum_{k\in\KK}(\iota_{C_k}+\sigma_{D_k}+h_k)(\scal{u}{o_k}).
\end{equation}
Let $w\in\HH$, let $(\alpha_k)_{k \in\KK}\in\RR^{\KK}$,
let $(\xi_k)_{k \in\KK}\in\ell^1(\KK)$, set
$\delta=\sqrt{\sum_{k \in\KK}\xi_k}$, and let 
\begin{equation}
\label{algoprox}
\begin{array}{l}
\text{for every}\;k\in\KK\\
\left\lfloor
\begin{array}{l}
\chi_k=\scal{w}{o_k}\\
\abs{\alpha_{k}}\leq  \displaystyle
\frac{\xi_k}{4 \gamma \max\{h_k (\abs{\chi_k}+2), 
h_k(-\abs{\chi_k}-2))\}
+2\abs{\chi_k}+1}\\[2ex]
\pi_{k}=\prox_{\gamma h_k}
\big(\soft{\gamma D_k} \chi_k \big)+ \alpha_{k}\\[1ex]
\nu_k=\proj_{C_k}\big(\sign (\chi_k)
\max\big\{0,\sign(\chi_k) \pi_{k}\big\}\big).
\end{array}
\right.\\
\end{array}
\end{equation}
Now set $v=\sum_{k\in\KK}\nu_ko_k$.
Then $v\simeq_\delta\prox_{\gamma G}w$.
\end{proposition}
\begin{proof}
The function $G$ lies in $\Gamma_0(\HH)$ as the composition of 
the linear isometry 
$\HH\to\ell^2(\KK)\colon u\mapsto(\scal{u}{o_k})_{k \in\NN}$
and the function
\begin{equation}
\ell^2(\KK)\to\RX\colon (\mu_k)_{k \in\KK} 
\mapsto \sum_{k \in\KK}g_k(\mu_k),
\quad\text{with}\quad g_k=\iota_{C_k}+\sigma_{D_k}+h_k,
\end{equation}
which belongs to $\Gamma_0(\ell^2(\KK))$ by Lemma~\ref{l:20151126a}.
Now set 
\begin{equation}
\label{eq:mumk0}
(\forall k\in\KK)\quad
\begin{cases}
\mu_{k}=\soft{\gamma D_k} \chi_k \\
s_{k} =\sign(\mu_k)\max\big\{0, \sign(\mu_k)(
\prox_{\gamma h_k} \mu_{k} +\alpha_{k})\big\} \\
\nu_k =\proj_{C_k}s_k.
\end{cases}
\end{equation}
Let $k \in\KK$. Since $\soft{\gamma D_k}$ is nonexpansive and 
$2\gamma\max\{h_k(\abs{\chi_k}+2),h_k(-\abs{\chi_k}-2)\}$ 
is a Lipschitz constant for $\gamma h_k$ 
on the interval $[-\abs{\chi_k}-1,\abs{\chi_k}+1]$, 
it follows from \eqref{eq:mumk0}
and Lemma~\ref{l:20150704a} that
\begin{equation}
\label{eq:inexcompoprox}
\begin{cases}
\delta_{k}^2 =\big(4 \gamma 
\max\{h_k(\abs{\chi_k}+2),h_k(-\abs{\chi_k}-2)\}
+2 \abs{
\chi_k} +1 \big) \abs{\alpha_{k}}\\
s_{k} \simeq_{{\delta}_{k}}\prox_{\gamma h_k}
(\prox_{\gamma\sigma_{D_k}}\chi_k)\\
s_{k}\prox_{\gamma \sigma_{D_k}} \chi_k\geq 0.
\end{cases}
\end{equation}
Thus, Lemma~\ref{lem:compprox} yields
\begin{equation}
s_{k}\simeq_{{\delta}_{k}}\prox_{\gamma(h_k+\sigma_{D_k})}\chi_k,
\end{equation}
and, using Lemma~\ref{lemma:inexprojprox}, we obtain
$\nu_k\simeq_{{\delta}_{k}}\prox_{\gamma g_k}\chi_k$.
Hence, by Definition~\ref{d:287s54},
\begin{equation}
\label{eq:ineqGLMalgo}
\gamma g_k(\nu_k)+\frac 1 2 \abs{\nu_k-
\chi_k}^2\leq  \gamma g_k\big(\prox_{\gamma g_k}
\chi_k\big)+\frac1 2 \abs{\prox_{\gamma g_k}
\chi_k-\chi_k}^2+\frac{\delta_{k}^2}{2}.
\end{equation}
On the other hand, we derive from 
\cite[Example~2.19]{Smms05} and
\cite[Proposition~3.6]{Siop07} that
\begin{equation} 
\label{eq:proxcomponent}
\scal{\prox_{\gamma G}w}{o_k}=\prox_{\gamma g_k} \chi_k.
\end{equation}
Thus, summing the inequalities \eqref{eq:ineqGLMalgo} over $k$, 
we obtain
\begin{align}
\label{eq:sumphik}
\nonumber\hskip -5mm &\gamma \sum_{k\in\KK} g_k(\nu_k)+\frac 1 2
\sum_{k\in\KK} \abs{\nu_k-\chi_k}^2\\
\nonumber 
&\leq\gamma \sum_{k\in\KK}\Big(g_k\big(
\scal{\prox_{\gamma g}w}{o_k}\big)+\frac1 2
\big\lvert \scal{\prox_{\gamma G}w}{o_k}-
\scal{w}{o_k}\big\rvert^2\Big)
+\frac 1 2 \sum_{k\in\KK} \delta_{k}^2\\
&\leq \gamma G\big(\prox_{\gamma G}w\big) 
+ \frac1 2 \norm{\prox_{\gamma G}w-w}^2 
+\frac 1 2 \sum_{k\in\KK} \xi_{k}\nonumber\\
&<+\infty.
\end{align}
Thus, \eqref{e:20151127a} and \eqref{eq:sumphik} yield
$(\nu_k)_{k \in\NN} \in\ell^2(\KK)$ and one can find
$v \in\HH$ such that, for every $k \in\KK$, $\chi_k=\nu_k$.
Hence,
\begin{equation}
\gamma G(v)+\frac 1 2 \norm{v-w}^2
\leq \gamma G(\prox_{\gamma G}w) +
\frac 1 2 \norm{\prox_{\gamma G}w-w}^2 +
\frac 1 2 \sum_{k\in\KK} \xi_{k} 
\end{equation}
and finally $v\simeq_{\delta} \prox_{\gamma G}w$, 
where $\delta=\sqrt{\sum_{k\in\KK}\xi_{k}}$.
\end{proof}

\begin{proof}[of Theorem~\ref{thm:algoconvergence}]
\ref{thm:algoconvergencei}: Lemma~\ref{l:20151126a} guarantees  
that $G\in\Gamma_0(\ell^2(\KK))$, that $G$ is coercive, and 
that $\dom G \subset \ell^r(\KK)$. The statement therefore
follows from \cite[Corollary~11.15(ii)]{Livre1}.

\ref{thm:algoconvergenceii}--\ref{thm:algoconvergenceiii}:
Let $\widehat{F}_n \colon \HH \to \RR\colon u \to (1/n)
\sum_{i=1}^n (A u (x_i)-y_i)^2$. 
Then, for every $u \in\ell^2(\KK)$, $\nabla \objfunc_n(u)=(2/ n)
\sum_{i=1}^n (\scal{u}{\Phi(x_i)}-y_i) \fmap(x_i)$.
Hence, since $\norm{\Phi(x_i)}_2\leq\kappa$,  $\nabla \objfunc_n$
is Lipschitz continuous with constant $2 \kappa^2$.
Therefore, the statement follows from Theorem~\ref{t:222} and 
Proposition~\ref{p:20150912a}. It remains to show the 
convergence properties of 
${(\norm{u_m-\widehat{u}}_r)}_{m \in\NN}$ and 
${(\norm{v_m-\widehat{u}}_r)}_{m \in\NN}$. We focus on the
sequence ${(\norm{u_m-\widehat{u}}_r)}_{m \in\NN}$, since 
${(\norm{v_m-\widehat{u}}_r)}_{m \in\NN}$ can be treated analogously.
It follows from Lemma~\ref{lem:totconvG} and the convexity of
$\widehat{F}_n$ that
\begin{equation}
\label{eq:20151127b}
(\forall\, m \in\NN)\quad (\widehat{F}_n+\lambda G)(u_m)-
(\widehat{F}_n+\lambda G)(\widehat{u}) \geq
\frac{\eta\lambda M
\norm{u_m-\widehat{u}}_r^2}{\big( \norm{\widehat{u}}_r 
+\norm{u_m-\widehat{u}}_r\big)^{2-r}}.
\end{equation}
Therefore, since $(\widehat{F}_n+\lambda
G)(u_m)-(\widehat{F}_n+\lambda G)(\widehat{u}) \to 0$ as $m \to
+\infty$ and $\psi\colon \RR_+ \to \RR \colon t
\mapsto t^2/(\norm{\widehat{u}}+t)^{2-r}$ is strictly increasing
with $\psi(0)=0$, we obtain $\norm{u_m-\widehat{u}}_r \to 0$.
Moreover, taking $\rho \in\RPP$ such that $\sup_{m \in\NN} \big(
\norm{\widehat{u}}_r+\norm{u_m-\widehat{u}}_r \big)^{2-r} \leq
\rho$, \eqref{eq:20151127c} follows from \eqref{eq:20151127b}.
\end{proof}




\appendix

\section{An auxiliary result}
\label{app:lsc}
The following result is a generalization of 
\cite[Proposition~5.14]{Smms05}.

\begin{lemma}
\label{l:20151126a}
Let $\KK$ be an at most countable set. For every $k\in\KK$,
let $C_k$ be a closed interval in $\RR$ such that $0\in C_k$,
let $D_k$ be a nonempty closed bounded interval in $\RR$,
and let $h_k\in\Gamma^+_0(\RR)$ be such that $h_k(0)=0$. Set
\begin{equation}
G\colon \ell^2(\KK)\to \RX\colon (\xi_k)_{k \in\KK} \mapsto \sum_{k
\in\KK} g_k(\xi_k),
\quad\text{where}\quad g_k=\iota_{C_k}+\sigma_{D_k}+h_k.
\end{equation}
Let $r \in\left]1,2\right]$ and consider the following statements:
\begin{enumerate}[label=\rm(\alph*)]
\item
\label{l:20151126a_a} 
$\sum_{k\in\KK}|(\inf D_k)_+|^2<\pinf$ and
$\sum_{k\in\KK}|(\sup D_k)_-|^2<\pinf$.
\item
\label{l:20151126a_b}
$\sum_{k\in\KK}h_k^*\big(-(\inf D_k)_+)<\pinf$
and $\sum_{k\in\KK}h_k^*((\sup D_k)_-)<\pinf$.
\item
\label{l:20151126a_c}
$\sum_{k\in\KK}|(\inf D_k)_+|^{r^*}<\pinf$ and
$\sum_{k\in\KK}|(\sup D_k)_-|^{r^*}<\pinf$.
\end{enumerate}
Then the following hold:
\begin{enumerate}
\item
\label{l:20151126ai}
Suppose that \ref{l:20151126a_a} or \ref{l:20151126a_b} is
satisfied. Then $G \in\Gamma_0(\ell^2(\KK))$.
\item
\label{l:20151126aii}
Suppose that \ref{l:20151126a_b} is satisfied.
Then $\inf\, G(\ell^2(\KK))>\minf$.
\item
\label{l:20151126aiii}
Suppose that, for every $k \in\KK$, $h_k \geq \eta \abs{\cdot}^r$ 
for some $\eta \in\RPP$. Then 
\ref{l:20151126a_a}$\Rightarrow$\ref{l:20151126a_c}%
$\Rightarrow$\ref{l:20151126a_b}.
\item
\label{l:20151126aiv}
Suppose that, for every $k\in\KK$, $h_k-\eta
\abs{\cdot}^r\in\Gamma^+_0(\RR)$ for some $\eta\in\RPP$ 
and that \ref{l:20151126a_c} holds.
Then, for every $\eta^\prime\in\left]0,\eta\right[$, there exists 
$H \in\Gamma_0(\ell^2(\KK))$ such that
$G\colon u\mapsto H(u)+\eta^\prime \sum_{k\in\KK}|\mu_k|^r$,
$\dom G\subset \ell^r(\KK)$, and $G$ is coercive in $\ell^2(\KK)$.
\end{enumerate}
\end{lemma}
\begin{proof}
We first observe that, if there exist 
$(\chi_k)_{k \in\KK}\in\ell^1_+(\KK)$ and $b\in\RP$ such that
\begin{equation}
\label{eq:20150704b}
(\forall k\in\KK)\quad-g_k\leq\chi_k+b \abs{\cdot}^2,
\end{equation}
then $G\in\Gamma_0(\ell^2(\KK))$.

\ref{l:20151126ai}:
Let $k\in\KK$. Since
\begin{equation}
(\forall \mu \in\RR)\quad 
\sigma_{D_k}(\mu)=
\begin{cases}
\mu \sup D_k & \text{if}\;\mu \geq 0\\
\mu \inf D_k & \text{if}\;\mu<0,
\end{cases}
\end{equation}
we have
\begin{align}
\label{eq:20151126c}
(\forall \mu \in\RR)\quad -g_k(\mu) 
&\leq-\sigma_{D_k}(\mu)-h_k(\mu)\nonumber\\
&\leq\max\{(\mu)_- (\inf D_k)_+, (\mu)_+ (\sup D_k)_-\}-h_k(\mu).
\end{align}
Hence, in order to guarantee condition \eqref{eq:20150704b} for 
some ${(\chi_k)}_{k \in\KK} \in\ell^1_+(\KK)$ and $b \in\RPP$,
it is sufficient to require condition \ref{l:20151126a_a} 
or \ref{l:20151126a_b} (note that $h_k^* \geq 0$, 
since $h_k(0)=0$).
Therefore in this case $G\in\Gamma_0(\ell^2(\KK))$.

\ref{l:20151126aii}:
It follows from \eqref{eq:20151126c} that
\begin{equation}
\label{e:20151127a}
(\forall\, k \in\KK)\qquad 
- g_k\leq\max\big\{ h_k^*\big( -(\inf D_k)_+\big), 
h_k^*\big( (\sup D_k)_- \big) \big\}.
\end{equation}
Hence, for every $u \in\ell^2(\KK)$, 
$- G(u)\leq\sum_{k \in\KK} \max\big\{ h_k^*\big( -(\inf D_k)_+\big), 
h_k^*\big( (\sup D_k)_- \big) \big\}<+\infty$.

\ref{l:20151126aiii}:
For every $k \in\KK$, 
$h_k^*\leq(r \eta)^{1-r^*} (r^*)^{-1} \abs{\cdot}^{r^*}$.
The statement therefore follows by observing that, 
since $2\leq r^*$, $\ell^2(\KK) \subset \ell^{r^*}(\KK)$.

\ref{l:20151126aiv}:
Setting, for every $k \in\KK$, $\tilde{h}_k=h_k-\eta^\prime \abs{\cdot}^r$, 
we have $g_k=\iota_{C_k}+\sigma_{D_k}+ \tilde{h}_k 
+ \eta^\prime \abs{\cdot}^r$, with
$(\eta -\eta^\prime) \abs{\cdot}^r\leq\tilde{h}_k \in\Gamma^+_0(\RR)$. It follows
from \ref{l:20151126ai} and \ref{l:20151126aiii} that, 
for every $u=(\mu_k)_{k\in\KK} \in\ell^2(\KK)$, 
$G(u)=H(u)+\eta^\prime \sum_{k \in\KK} \abs{\mu_k}^r$,
for some $H \in\Gamma_0( \ell^2(\KK))$.
\end{proof}

\section{Proximity operators of power functions}
\label{app:proxr}
It follows from
\cite[Example 4.4]{ChaCombPesqWajs07} that, for every 
$\gamma\in\RPP$ and every $r\in [1,2]$, 
\begin{equation}
\label{eq:proxpower}
(\forall\, \mu \in\RR)\quad\prox_{\gamma
\abs{\cdot}^r}\mu=\xi\sign(\mu),\quad \text{where}\quad \xi
\geq 0 \quad\text{and}\quad \xi+r \gamma \xi^{r-1}=\abs{\mu}.
\end{equation}
There are several exponents $r$ for which 
Equation \eqref{eq:proxpower} can be solved explicitly for
$r\in\{3/2,4/3,5/4\}$ \cite{ChaCombPesqWajs07,Wajs2007}.
However, in general, it must be solved iteratively.

\begin{proposition}
\label{prop:powerprox}
Let $\mu\in\RR$, let $\gamma\in\RPP$, let $r\in [1,2]$, and let
$(r_1,r_2)\in [1,2]^2$, be such that $r_1<r_2$. 
Then the following hold:
\begin{enumerate}
\item
\label{item:powerprox0} 
$\prox_{\gamma \abs{\cdot}^r}\colon \RR \to \RR$ is strictly
increasing, nonexpansive, odd, and differentiable,
and $\prox_{\gamma \abs{\cdot}^r}+\iota_{\RP}$ is convex.
\item
\label{item:ineqproxpower} 
We have
\begin{equation}\label{ineq:proxpower}
\min\bigg\{ \frac{\abs{\mu}}{1+r \gamma}, 
\Big( \frac{\abs{\mu}}{1+r
\gamma} \Big)^{\frac {1}{r-1}} \bigg\}\leq
\abs{\prox_{\gamma \abs{\cdot}^r}\mu} \leq
\max\bigg\{ \frac{\abs{\mu}}{1+r \gamma}, 
\Big( \frac{\abs{\mu}}{1+r \gamma} \Big)^{\frac {1}{r-1}} 
\bigg\}.
\end{equation}
\item
\label{ineq:proxpower2} 
Suppose that $\abs{\mu}>1+r_2 \gamma$. Then 
$\abs{\prox_{\gamma \abs{\cdot}^{r_2}}\mu} 
< \abs{\prox_{\gamma\abs{\cdot}^{r_1}}\mu}$.
\item
\label{ineq:proxpower3} 
Suppose that $r>1$ and that $\abs{\mu}>1+r \gamma$. Then 
$\dfrac{\abs{\mu}}{1+r \gamma}\leq \abs{\prox_{\gamma
\abs{\cdot}^{r}}\mu}<\abs{\mu}-\gamma$.
\end{enumerate}
\end{proposition}
\begin{proof}
\ref{item:powerprox0}: It follows from \cite[Lemma~2.2(iv) and 
Proposition~2.4]{Siop07} that $\prox_{\tau \abs{\cdot}}$ is
nonexpansive, increasing, and odd. Now set $\psi\colon \RP \to
\RP\colon\xi\mapsto\xi+r\tau \xi^{r-1}$. Clearly $\psi$ is
strictly increasing and concave. Moreover it is differentiable on
$\RPP$ and, for every $\xi\in\RPP$, $\psi^\prime(\xi)=1 +
r(r-1)/\xi^{2-r}$. Hence, from \eqref{eq:proxpower}, for every
$\mu\in\RP$, $\prox_{\tau \abs{\cdot}^r}\mu=\psi^{-1}(\mu)$.
This shows that $\prox_{\tau\abs{\cdot}^r}$ is strictly
increasing, convex, differentiable on $\RPP$ with, for every 
$\mu\in\RPP$, $(\prox_{\tau\abs{\cdot}^r})^\prime\mu
=1/\psi^\prime(\psi^{-1}(\mu))$, that is
\begin{equation}
\label{eq:powerproxdiff}
(\prox_{\gamma\abs{\cdot}^r})^\prime\mu
=\bigg(1+\frac{r(r-1)\gamma}{(\prox_{\gamma
\abs{\cdot}^r}\mu)^{2-r}}\bigg)^{-1}.
\end{equation}

\ref{item:ineqproxpower}: 
According to \eqref{eq:proxpower}, there exists $\xi\in\RP$ such 
that $\prox_{\tau\abs{\cdot}^{r}}\mu=\sign(\mu) \xi$ and  
$\xi+r \tau \xi^{r-1}=\abs{\mu}$. If $\xi\geq 1$, 
then $\abs{\mu}=\xi+r \tau \xi^{r-1}\leq (1+r \tau) \xi$, hence
$\abs{\mu}/(1+r \tau)\leq \xi=\abs{\prox_{\tau
\abs{\cdot}^r}\mu}$. If $\xi<1$, then $\abs{\mu}=\xi+r \tau
\xi^{r-1}\leq (1+r \tau) \xi^{r-1}$, hence $\big(\abs{\mu}/(1 +
r \tau) \big)^{1/(r-1)}\leq \xi=\abs{\prox_{\tau
\abs{\cdot}^r}\mu}$. The first inequality in
\eqref{ineq:proxpower} follows and the second is 
proved analogously.

\ref{ineq:proxpower2}:
In view of \eqref{eq:proxpower} there exist
$\xi_1\in\RP$ and $\xi_2\in\RP$ such that
\begin{equation}
\begin{cases}
\prox_{\tau\abs{\cdot}^{r_1}}\mu=\sign(\mu) \xi_1
\quad\text{and}\quad \xi_1 +
r_1 \tau \xi_1^{r_1-1}=\abs{\mu}\\
\prox_{\tau\abs{\cdot}^{r_2}}\mu=\sign(\mu) \xi_2
\quad\text{and}\quad \xi_2 +
r_2 \tau \xi_2^{r_2-1}=\abs{\mu}.
\end{cases}
\end{equation}
If $\abs{\mu}>1+\tau r_2>1+\tau r_1$, 
it follows from \eqref{ineq:proxpower} that 
\begin{equation}
1<\frac{\abs{\mu}}{1+r_1 \tau}\leq \abs{\xi_1}\quad \text{and}\quad 
1<\frac{\abs{\mu}}{1+r_2 \tau}\leq \abs{\xi_2}.
\end{equation}
Therefore, since $r_1<r_2$ and $\xi_1> 1$,
\begin{equation}
\xi_2+r_2 \tau \xi_2^{r_2-1}=\abs{\mu}=\xi_1+r_1 \tau \xi_1^{r_1-1}
<\xi_1+r_2 \tau \xi_1^{r_2-1}.
\end{equation}
Hence, since $\xi\mapsto\xi+r_2\tau\xi^{r_2-1}$ 
is strictly increasing on $\RP$, we conclude that $\xi_2<\xi_1$.

\ref{ineq:proxpower3}:
Since \eqref{eq:proxpower} implies that
$\prox_{\tau\abs{\cdot}}\mu=\sign(\mu) (\abs{\mu}-\tau)$,
we derive from \ref{ineq:proxpower2} that
\begin{equation}
\abs{\mu}>1+r\tau\quad\Rightarrow\quad
\abs{\prox_{\tau\abs{\cdot}^{r}}\mu}<\abs{\mu}-\tau,
\end{equation}
The first inequality in \ref{ineq:proxpower3} follows directly from
\eqref{ineq:proxpower}.
\end{proof}

\begin{remark}\ 
\begin{enumerate}
\item 
The bounds given in \eqref{ineq:proxpower} can be
useful to initialize the bisection method to solve 
\eqref{eq:proxpower}.
\item 
$(\prox_{\gamma \abs{\cdot}^r})^\prime 0 =0$,
$(\prox_{\gamma \abs{\cdot}^r})^\prime \mu\leq 1$ and
$(\prox_{\gamma \abs{\cdot}^r})^\prime \mu  \to 1$ as $\mu \to
+\infty$.
\item 
$\prox_{\gamma \abs{\cdot}^r}$ has no asymptote as 
$\mu\to+\infty$, since \eqref{eq:proxpower} yields
$\prox_{\gamma \abs{\cdot}^r}\mu-\mu=-r \gamma 
(\prox_{\gamma\abs{\cdot}^r}\mu)^{r-1}\to\minf$ as
$\mu \to +\infty$.
\end{enumerate}
\end{remark}

\end{document}